\documentclass[11pt, letterpaper]{article}
\usepackage{color}
\usepackage{amsmath,amsthm,amsfonts,amssymb}
\newtheorem{theorem}{Theorem}

\newtheorem{proposition}[theorem]{Proposition}

\usepackage[letterpaper,margin=1in]{geometry}

\def\RR{{\mathbb R}}

\def\EE{{\mathbb E}}

\def\PP{{\mathbb P}}

\renewcommand{\emptyset}{\varnothing}

\DeclareMathOperator*{\circu}{circ}
\DeclareMathOperator*{\argmax}{arg \ max}

\title{Dynamic Relaxations for Online Bipartite Matching}

\author{Alfredo Torrico \qquad Alejandro Toriello\\
\small{H.\ Milton Stewart School of Industrial and Systems Engineering}\\
\small{Georgia Institute of Technology}\\
\small{Atlanta, Georgia 30332}\\
\small{atorrico3 at gatech dot edu, atoriello at isye dot gatech dot edu}
}

\begin{document}

\maketitle

\begin{abstract}
\emph{Online bipartite matching} (OBM) is a fundamental model underpinning many important applications, including search engine advertisement, website banner and pop-up ads, and ride-hailing. We study the i.i.d.\ OBM problem, where one side of the bipartition is fixed and known in advance, while nodes from the other side appear sequentially as i.i.d.\ realizations of an underlying distribution, and must immediately be matched or discarded. We introduce \emph{dynamic} relaxations of the set of achievable matching probabilities, show how they theoretically dominate lower-dimensional, \emph{static} relaxations from previous work, and perform a polyhedral study to theoretically examine the new relaxations' strength. We also discuss how to derive heuristic policies from the relaxations' dual prices, in a similar fashion to dynamic resource prices used in network revenue management. We finally present a computational study to demonstrate the empirical quality of the new relaxations and policies.
\end{abstract}

\section{Introduction}

Many important and emerging applications in e-commerce and in the internet more generally can be modeled as online two-sided markets, with buyers and sellers dynamically appearing and conducting transactions. When a platform or other entity controls or manages one side of this market (usually the supply) and can choose what product to offer to dynamically appearing buyers, the system in question can be modeled as an \emph{online bipartite matching} (OBM) problem. 
As more services move to online platforms in the coming years, the ubiquity and importance of OBM models will only increase. 

An important application of OBM and its generalizations is in the rapidly growing sector of digital advertisement; in their \emph{US Ad Spending Estimates and Forecast for 2017}, eMarketer reports that digital ad spending reached \$83 billion last year, an almost 16\% year-over-year increase. Within digital marketing, search engine advertisement yields one application of OBM and similar models \cite{mehta13,mehta_etal07}: Users input search terms, and the engine displays ad links in addition to the actual search results. The engine chooses the ad(s) to display (i.e.\ matches an ad to a user) based on the search term, user information, and advertisers' preferences and budget, with one typical objective being to maximize the expected revenue collected from advertisers. Similarly, OBM models can be applied to website banner and pop-up ads; here, each time a user loads a website, the site manager can choose ad(s) to display based on the user's information and browsing history as well as the advertisers' budget and target market.

OBM also finds applications in ride-hailing \cite{raghvendra16}, another rapidly growing sector -- one study by Goldman Sachs in 2017 projected that global revenues in the industry could reach \$285 billion by 2030\footnote{C.\ Huston. ``Ride-hailing industry expected to grow eightfold to \$285 billion by 2030.'' \emph{MarketWatch}, published May 27, 2017.}. Within these systems, when a user requests a ride, the ride-hailing platform must match them to an available driver, with the overall goal of maximizing some measure of customer satisfaction or utility (for example, by minimizing users' average waiting time before a pickup).

As in classical deterministic bipartite matching, OBM involves matching nodes on opposite sides of a bipartite graph, with the objective of maximizing the cardinality of the matching or a more general weight function. In online versions of the problem, nodes on one or both sides of the bipartition may appear and/or disappear dynamically, matches are often irrevocable, and decisions must usually be made with only partial information about the underlying graph. In the version of OBM we study here, one side of the bipartition is fixed and known, representing the goods or resources the platform can offer; the nodes from the other side, representing customers, arrive sequentially, one at a time. Upon each arrival, the platform must immediately and irrevocably match the arriving node to a remaining compatible node from the other side, or discard it. We assume arriving nodes are i.i.d.\ draws from an underlying uniform distribution over possible node ``types'', representing customer classes that may or may not be compatible with different resources or goods. For example, in search engine advertisement, advertisers indicate which search terms they wish their ads displayed with, and the search engine can only match ads with terms in these classes. The i.i.d.\ assumption implies this model is applicable in situations where the platform can forecast customer behavior, e.g.\ based on past arrival data, and where this behavior is relatively stable over time. The model may not be as applicable in data-poor situations where the platform cannot confidently forecast customer behavior; the literature includes more conservative models for such cases \cite{mehta13}, culminating with the adversarial model studied in \cite{karp_etal90}. Conversely, if customer behavior can be forecast but is not necessarily stable over time, the assumption of identical distributions may be problematic. While we are not aware of OBM models for this case, the revenue management literature includes many works in this vein, particularly in network revenue management, e.g.\ \cite{talluri_van_ryzin98}. We briefly survey related work below.

Perhaps because of its applicability in search engine advertisement, the algorithms community has extensively studied i.i.d.\ OBM and related models for the past decade, starting with \cite{feldman_etal09}. For the i.i.d.\ variant with cardinality objective, this work typically focuses on developing heuristic matching policies with multiplicative worst-case performance guarantees. It is straightforward to see that a simple policy based on solving a max-flow relaxation achieves at least $ 1 - 1/e \approx 0.63 $ of the optimal policy, and this ratio in fact matches the best possible competitive ratio in the adversarial case \cite{karp_etal90}; \cite{feldman_etal09} established that in the less conservative i.i.d.\ model a better ratio is indeed possible. Currently, the best guarantee of this type is roughly $ 0.71 $ \cite{jaillet_lu14} (and it is slightly better under additional assumptions).

The analysis of these heuristic policies and their worst-case guarantees relies on simple linear programming (LP) relaxations, often with network flow structure. However, there is comparably far less work focusing directly on the derivation of strong relaxations for i.i.d.\ OBM, even though these relaxations provide dual upper bounds useful for benchmarking any new heuristic policies and often can be employed in policy design as well. The only work along these lines we are aware are of is  \cite{torrico_etal18}, which builds on the network flow relaxations used in \cite{feldman_etal09}, and subsequent papers, by adding more sophisticated valid inequalities derived from probabilistic arguments. Though somewhat unusual in the algorithms or optimization literature, this approach can be interpreted as a version of \emph{achievable region} techniques from queueing theory and applied probability; see e.g.\ \cite{bertsimas_nino-mora96,coffman_mitrani80}.

\subsection{Contribution}

Starting with \cite{feldman_etal09}, to the best of our knowledge all known relaxations for i.i.d.\ OBM are ``static'': Although the process occurs dynamically over a horizon with sequential decision epochs, the relaxations use as their primary variables the probability that an arriving customer node of some type is ever matched to a fixed resource node. Though this reduces the number of variables to consider, it also means the corresponding relaxations are coarser and looser, as they cannot easily capture the model's dynamics. Furthermore, with few exceptions, the policies derived from such relaxations are also mostly static in nature; that is, though a decision may depend on the arriving node type and the remaining available resource nodes, it usually does not depend on the decision epoch itself and how far or close it might be from the end of the horizon.

Our main contribution is to explicitly account for the problem's sequential nature and consider dynamic relaxations. Specifically, these relaxations use as decision variables the probability that a particular match occurs \emph{in a particular stage}. Using these time-indexed probabilities affords several modeling advantages, such as allowing us to include edge weights that vary by time and thus simplify the analysis by capturing all compatibility information in the objective.

Of course, the primary appeal of dynamic relaxations is the possibility of providing tighter dual bounds for the model. As one of our main results, we establish that our simplest dynamic relaxation is provably at least as tight as the best-performing relaxation from \cite{torrico_etal18}; furthermore, the latter relaxation includes exponentially many inequalities and relies on a separation algorithm, whereas our new relaxation has polynomially many variables and constraints. To further understand our new relaxation, we also perform a polyhedral study, demonstrating that all of its inequalities are facet-defining for the underlying polytope of achievable probabilities. We then extend this polyhedral study and introduce more complex inequalities, all facet-defining as well. Our empirical study verifies the strength of the new relaxation; it improves the previous best gaps by $4\%$ to $5\%$ in absolute terms on average.

As a secondary contribution, we also show how our new relaxation can be leveraged to construct a dynamic heuristic policy. Although this kind of policy is new in OBM to our knowledge, our policy can be viewed as the OBM analogue to dynamic bid price policies, introduced in \cite{adelman07} for network revenue management. To design the policy, we establish a connection between our relaxation and a value function approximation of the model's dynamic programming (DP) formulation. Our empirical results also verify the new policy's quality in comparison to the best empirically performing policies from the literature.

The remainder of the paper is organized in the following way: After a brief literature survey at the end of this section, Section \ref{sec:desc} formulates the problem, and summarizes pertinent previous results. Section \ref{sec:relax} introduces our relaxations and gives our theoretical results, while Section \ref{sec:exp} outlines our computational study. Section \ref{sec:conc} concludes and discusses possible future work. An Appendix has mathematical proofs not included in the body of the paper.

\subsection{Literature Review}\label{sec:related_work}

The OBM model was introduced by \cite{karp_etal90}, who studied the adversarial version in which node arrivals are not governed by a distribution, but rather by an adversary whose objective is to maximize the difference between the cardinality of the decision maker's matching and the offline optimum. The authors show that a randomized \emph{ranking} algorithm that chooses a random permutation of the resource nodes and matches the highest-ranked available and compatible node according to the permutation yields an optimal competitive ratio of $ 1 - 1/e $; see e.g.\ \cite{birnbaum_mathieu08} for a simplified and corrected proof.

Most work on OBM since has focused on less conservative variants. The i.i.d.\ version with cardinality objective was first studied in \cite{feldman_etal09}, who showed that in this version the performance guarantee could be strictly better than $ 1 - 1/e $; they also used a network flow relaxation to design their \emph{two suggested matchings} policy. Subsequent work has focused on improving the performance guarantee and/or generalizing the objective, and typically relies on LP relaxations with network flow structure, e.g.\ \cite{bahmani_kapralov10,brubach_etal17,haeupler_etal11,manshadi_etal12}. For the cardinality case, \cite{jaillet_lu14} have the current state of the art, a policy based on a max-flow relaxation with a guarantee of approximately $ 0.71 $.

The i.i.d.\ version of OBM (sometimes also called the \emph{known i.i.d.}\ model) is in some sense the least conservative OBM variant, compared to the most conservative adversarial version. Some authors have studied models that compromise between the two. For example, in the \emph{random permutation} model an adversary chooses the graph, but the arriving nodes are revealed in a random order not controlled by the adversary. With this slight relaxation of the adversarial framework, \cite{goel_mehta08} show that for the cardinality objective a simple greedy algorithm, which matches an arriving node to any remaining compatible node, achieves a competitive ratio of $1-1/e$; later improvements came in \cite{karande_etal11,mahdian_yan11}. Other models, variants and extensions have appeared in the algorithms literature; we refer the reader to the survey \cite{mehta13}.

While the notion of dynamic relaxations appears to be new in the OBM context, there is a stream of related literature in network revenue management, beginning with \cite{adelman07}, who introduced \emph{dynamic bid relaxations} and their corresponding policies. In this literature, the goal is often to show that a particular relaxation can be computed efficiently, e.g.\ \cite{kunnumkal_talluri16,tong_topaloglu14,vossen_zhang15}, as a naive formulation involves a separation problem solved via an integer program. These dynamic relaxations have also been extended to customer choice models, e.g.\ \cite{vossen_zhang15,zhang_adelman09}.

\section{Model Description and Preliminaries}\label{sec:desc}

The OBM model is formulated using two finite disjoint sets $N$ and $V$, with the process occurring dynamically in the following way. The \emph{right-hand} node set $V$, with $ \lvert V \rvert =m$, is known and given ahead of time. 
The \emph{left-hand} set $N$ with $ \lvert N \rvert = n $ represents different node types that may appear, but we do not know which ones will appear and how often. We know only that $ T $ left-hand nodes in total will appear sequentially, each one drawn independently from the uniform distribution over node types $ N $. That is, at each epoch a node from one of the types in $N$ appears with probability $1/n$ and must be immediately (and irrevocably) matched to a remaining available node in $V$ or discarded; two or more nodes from the same type may appear throughout the process, each treated as a separate copy.
Matching $ i \in N $ to $ j \in V $ in stage $ t $ yields a (known) reward or weight $ w_{ij}^t $, and the objective is to maximize the expected weight of the matching. Following convention from previous literature and the motivating application of search engine advertisement, we call $i\in N$ an \emph{impression}, and each $j\in V$ an \emph{ad}. 

By considering time-indexed weights $w_{ij}^t$, we generalize much of the existing literature and can avoid dealing with specific graph structure. 
In particular, we may assume that the process occurs in a complete bipartite graph, i.e.\ every node type in $N$ is connected or compatible with every node in $V$; non-existent edges simply get weight zero. 

Moreover, we can assume $m=n=T$ without loss of generality. Indeed, if $m<T$ we add dummy nodes to $V$ and assign zero weight to all corresponding edges. Similarly, if $m>T$ we  increase the number of  stages  and give zero weight to all edges in the new stages. 
If $n>m = T$, we again add dummy nodes and stages.
Finally, if $n< m = T$ we  make $\kappa$ copies of every node type in $N$ (and the corresponding edges) for the smallest $ \kappa $ with $ \kappa n \geq m$, then proceed as before. To ease notation, in the remainder of the paper we  write $n$ for $m$ and $T$, but we use the indices $i$ for impressions, $j$ for ads, and $t$ for stages. We use the shorthand $ [n] := \{1, \dotsc, n\} $, and identify singleton sets with their unique element.

\subsection{DP and LP Formulations}
Let $\eta$ be the random variable with uniform distribution over $N$. We count stages down from $n$, meaning stage $t$ occurs when $t$ decision epochs (including the current one) remain in the process. We can now give a DP formulation for this OBM model. Let $v^*_t(i,S)$ denote the optimal expected value given that $i\in N$ appears in stage $t$ when the set of ads $S\subseteq V$ is available. Then, for all $t=1,\dotsc,n $, $ i\in N $ and $ S\subseteq V$,
\begin{align}\label{eq:DP}
v^*_t(i,S) & = \max
\begin{cases}
\max_{j\in S}\{w^t_{ij}+\EE_{\eta}[v^*_{t-1}(\eta,S\backslash j)]\} \\
\EE_{\eta}[v^*_{t-1}(\eta,S)],
\end{cases}  
\end{align}
where $v_0^*(\cdot,\cdot)$ is identically zero, and the optimal expected value of the model is given by $\EE_{\eta}[v^*_{n}(\eta,V)] = 1/n \sum_{i \in N} v_n^*(i, V)$. The first term in this recursion corresponds to matching $i$ with one of the remaining ads $j\in S$; the second corresponds to discarding $i$. As with any DP, the optimal value function $v^*$ induces an optimal policy: At any state $(t,i,S)$, we choose an action that attains the maximum in \eqref{eq:DP}. 

Using a standard reformulation (see e.g.\ \cite{puterman05}), we can capture the recursion \eqref{eq:DP} with the linear program
\begin{subequations}\label{eq:obm_lp_val}
\begin{align}
\min_{ v \geq 0 } ~ & \EE_{\eta}[v_n(\eta,V)] \label{eq:obm_lp_valobj} \\
\text{s.t.\ } & v_{t}(i,S\cup j)-\EE_{\eta}[v_{t-1}(\eta,S)]\geq w^t_{ij}, & & t\in[n], \ i\in N, \ j\in V, \ S\subseteq V\backslash j \label{eq:obm_lp_valx}\\
& v_{t}(i,S)-\EE_{\eta}[v_{t-1}(\eta,S)]\geq 0,  & & t\in[n], \ i\in N, \ S\subseteq V .  \label{eq:obm_lp_valy}
\end{align}
\end{subequations}
The value function $v^*$ defined in  (\ref{eq:DP}) is optimal for \eqref{eq:obm_lp_val}. Moreover, this LP is a strong dual for OBM; any feasible $v$ has an objective greater than or equal to $\EE_{\eta}[v^*_{n}(\eta,V)]$. The dual of \eqref{eq:obm_lp_val} is a  primal formulation where any feasible solution encodes a feasible policy and its probability of choosing any action from any state in the DP. That formulation is the LP
\begin{subequations}\label{eq:obm_lp_pol}
\begin{align}
\max_{x,y \geq 0} & \quad   \sum_{i\in N} \sum_{j\in V} \sum_{t\in [n]} \sum_{S\subseteq V\backslash j} w^t_{ij}x_{i,j}^{t,S}\\
\text{s.t.\ } & \sum_{j\in V}x_{i,j}^{n,V\backslash j}+y_i^{n,V}\leq \frac{1}{n}, \quad i\in N, \label{eq:obm_lp_pol1} \\
\begin{split}
 & \sum_{j\in S}x_{i,j}^{t,S\backslash j}  +y_i^{t,S}\cdot1_{\{t\neq 1\}}-\frac{1}{n}\cdot1_{\{|S|> t\}}\cdot\sum_{k\in N}y_{k}^{t+1,S} \\
& \qquad -\frac{1}{n}\sum_{k\in N}\sum_{j\in V\backslash S}x_{k,j}^{t+1,S}\leq0,   \qquad t\in[n-1], \ i\in N, \ \emptyset\neq S\subset V, \ |S|\geq t,
\end{split} \\
& \sum_{j\in V} x_{i,j}^{t,V\backslash j} +y_i^{t,V}\cdot1_{\{t\neq1\}} -\frac{1}{n}\sum_{k\in N}y_k^{t+1,V}\leq 0, \qquad  i\in N, \ t\in[n-1]. \label{eq:obm_lp_pol2} 
\end{align}
\end{subequations}
We denote by $1_{\mathcal{A}}$ the indicator function for a condition $\mathcal{A}$, which takes value one if condition $\mathcal{A}$ is satisfied, zero otherwise. Decision variable $ x_{i,j}^{t,S} $ 
represents the probability that the policy chooses to match impression $ i $ to ad $ j $ in state $ ( t, i, S \cup j ) $, and $ y_i^{t,S} $
similarly represents a discarding action.

As with its dual, \eqref{eq:obm_lp_pol} has exponentially many variables and constraints, and is therefore difficult to analyze directly. However, we can equivalently consider the probability that a feasible policy makes a particular match between $ i $ and $ j $ in stage $ t $ without tracking the other remaining ads $ S \subseteq V \setminus j $; this corresponds to optimizing over a projection of the feasible region of \eqref{eq:obm_lp_pol},
\begin{align*}
\max \Biggl\{ \sum_{ i \in N } \sum_{ j \in V}\sum_{t\in[n]} w^t_{ij}z^t_{ij} : \exists ~ (x, y) \geq 0 \text{ satisfying } \eqref{eq:obm_lp_pol1}\text{--}\eqref{eq:obm_lp_pol2} \text{ with } z^t_{ij} =\sum_{S\subseteq V\backslash j}x_{i,j}^{t,S} \Biggr\} ,
\end{align*}
where $ z^t_{ij} $ is the probability that impression $ i $ is matched to ad $ j $ in stage $ t $. Any such $ z $ is a vector of matching probabilities that is \emph{achievable} by at least one feasible policy. Let $ Q $ denote this projected polyhedron in the space of $ z_{ij}^t $ variables, and note that $ Q $ is full-dimensional in $ \RR^{n^3}$. Optimizing over $Q$ is as difficult as solving the original DP formulation \eqref{eq:DP}, but optimizing over any relaxation of $Q$ yields a valid upper bound; this is our main goal.

\subsection{Relevant Previous Work}

Most previous results concerning relaxations for OBM use a lower-dimensional projection of the feasible region of \eqref{eq:obm_lp_pol}. Specifically, assuming edge weights are static across stages, $ w_{ij}^t = w_{ij} $ for $ t \in [n] $, consider
\begin{align*}
\max \Biggl\{ \sum_{ i \in N } \sum_{ j \in V} w_{ij} z_{ij} : \exists ~ (x, y) \geq 0 \text{ satisfying } \eqref{eq:obm_lp_pol1}\text{--}\eqref{eq:obm_lp_pol2} \text{ with } z_{ij} =\sum_{t\in[n]}\sum_{S\subseteq V\backslash j}x_{i,j}^{t,S} \Biggr\} ,
\end{align*}
where $ z_{ij} $ is the probability that impression $ i $ is ever matched to ad $ j $. Let $ Q' $ denote this projected polyhedron in the space of $z_{ij}$ variables, and observe that $ Q' $ is also a projection of $Q$ via $z_{ij}=\sum_{t\in[n]}z_{ij}^t$. The following max flow (or deterministic bipartite matching) LP is known to be a relaxation of $Q'$ and has been used to study $Q'$ in several works starting with \cite{feldman_etal09}: 
\begin{subequations}\label{eq:obm_relax_prev}
\begin{align}
\max_{z \geq 0} ~ & \sum_{i \in N} \sum_{j \in V} w_{ij} z_{ij} \\
\text{s.t.\ } & \sum_{j\in V}z_{ij}\leq T/n=1, \qquad i\in N \label{eq:obm_relax_prev_icap} \\
& \sum_{i\in N}z_{ij}\leq 1, \qquad j\in V . \label{eq:obm_relax_prev_jcap} 
\end{align}
\end{subequations}
In this relaxation, constraints \eqref{eq:obm_relax_prev_icap} limit the expected number of times an impression type can be matched to $ T/n = 1 $, the expected number of times it will appear, while \eqref{eq:obm_relax_prev_jcap} state that each ad is matched at most once. 

To our knowledge, the only past work that specifically focuses on polyhedral relaxations of $Q'$ is \cite{torrico_etal18}, which presents several classes of valid inequalities, including the \emph{right-star} inequalities,
\begin{align}\label{eq:right-star_prev}
\sum_{i\in I}z_{ij}\leq 1- (1- \lvert I \rvert /n )^n, \qquad j\in V, \ I\subseteq N,
\end{align}
which yield the best empirical bounds when added to \eqref{eq:obm_relax_prev}.
Although exponential in number, these inequalities can be separated over in polynomial time by a simple greedy algorithm. We use the bound given by \eqref{eq:obm_relax_prev} with \eqref{eq:right-star_prev} as a theoretical and empirical benchmark to test our new relaxations.

\section{Dynamic Relaxations}\label{sec:relax}

We introduce various classes of valid inequalities for $Q$ and study their facial dimension. These inequalities always include variables corresponding to complete bipartite subgraphs; therefore, to ease notation we define
$$Z_{I,J}^{t} := \sum_{i\in I}\sum_{j\in J}z_{ij}^{t}, \qquad I \subseteq N, J \subseteq V .$$
We begin by presenting a simple inequality class to motivate our approach. For an impression $i\in N$, the probability of matching $i$ in each stage $t\in[n]$ is at most $1/n$; this corresponds to
\begin{equation}\label{eq:prob_bound}
Z_{i,V}^t\leq 1/n, \qquad  \ i\in N, \ t\in [n].
\end{equation}
Note that by summing these constraints over all $ t $ for a fixed $ i $, we obtain \eqref{eq:obm_relax_prev_icap}.
\begin{proposition}\label{prop:facetProbBoundExt}
Constraints \eqref{eq:prob_bound} are facet-defining for the polyhedron of achievable probabilities $Q$.
\end{proposition}
\begin{proof}Fix $i\in N$ and $t\in [n]$. We use $e_{k,j}^\tau \in\RR^{n^3}$ to denote the canonical vector, i.e., a vector with a one in the coordinate $(k,j,\tau)$ and zero elsewhere, indicating that we match impression $k$ with ad $j$ in stage $\tau$. We construct the following $n^3$ affinely independent points corresponding to policies that satisfy \eqref{eq:prob_bound} with equality:
\begin{itemize}
\item Policy for $(i,j,t)$ with $j\in V$: If $i$ appears in stage $t$, which happens with probability $1/n$,  we match it with $j$. This corresponds to the point $\frac{1}{n}e_{i,j}^t$.
\item Policy for $(k,j,\tau)$ with $j\in V$, $\tau \neq t$, and $k\in N$: If $k$ appears in stage $\tau$ (with probability $1/n$), we match it to $ j $. Then, if $i$ appears in stage $t$ with probability $1/n$, we match it to some $\ell \in V$, $\ell\neq j$, so we have the point $\frac{1}{n}e_{k,j}^\tau +\frac{1}{n}e_{i,\ell}^t$.
\item Policy for $(k,j,t)$ with $j\in V$, and $k\neq i$: If $k$ appears in stage $t$ (with probability $1/n$), we match it to $j$. On the other hand, if $i$ appears in stage $t$ with probability $1/n$, we match it to some $\ell\in V$, $\ell \neq j$, so we have the point $\frac{1}{n}e_{k,j}^t+\frac{1}{n}e_{i,\ell}^t$.
\end{itemize}
These points are linearly independent, which implies they are affinely independent.
\end{proof}

We now introduce our general inequality family. Fix a set of ads $J\subseteq V$ and a family of impression sets $ I_t \subseteq N $, $ t \in [n] $. For any vector $\alpha\in\RR^n_+$, we have a valid inequality for $Q$ of the form
\begin{equation}\label{eq:dp_ineq}
\sum_{t=1}^n \alpha_t Z_{I_t,J}^t \leq R(\alpha, (I_t), J ), 
\end{equation}
where $R(\alpha, (I_t), J )$ defines the maximum of the left-hand side over $Q$. As one example, inequalities \eqref{eq:prob_bound} are a special case of \eqref{eq:dp_ineq} where $ J = V $, $ I_t = \{i\} $, $ \alpha_t = 1 $, and $ I_\tau = \emptyset $, $ \alpha_\tau = 0 $ for $ \tau \neq t $.
\begin{proposition}\label{prop:dp_ineq}
For $\alpha\in\RR^n_+$, set family $I_t\subseteq N$, $ t \in [n] $, and $J\subseteq V$, constraints \eqref{eq:dp_ineq} are valid for the polyhedron of achievable probabilities $Q$, and $R(\alpha, (I_t), J )$ can be computed in polynomial time via a DP.
\end{proposition}
\begin{proof}
Define variables $p_t\in\{0,1\}$ to indicate whether a node from $I_t$ appears in stage $t$ or not, and denote by $d\in\{0,\ldots,|J|\}$ the number of remaining nodes from $J$. Given this, we can state a DP recursion using the value function $R(t,d,p_t)$, the expected value in stage $t$ when $d$ nodes from $J$ are available and $ p_t $ has occurred. For example, if only one stage remains, $d$ nodes from $J$ are available, and no element of $I_1$ appears, $R(1,d,0)=0$ since we cannot match any node in $ I_1 $. Conversely, $R(1,d,1)= \alpha_1 \min\{1, d \}$, since we can match a node and obtain value $ \alpha_1 $ as long as at least one element of $ J $ remains.

In general, if $d$ nodes are available in stage $t$ and no node from $I_t$ appears ($p_t=0$), then the expected value $R(t,d,0)$ can be computed recursively by conditioning on terms from stage $ t - 1 $: 
\[
R(t,d,0) = \frac{ n - \lvert I_t \rvert }{n} R( t-1, d, 0 ) + \frac{\lvert I_t \rvert}{n} R( t-1, d, 1 ).
\]
On the other hand, to compute $R(t,d,1)$ we choose the maximum between discarding or matching, with value
$$
R(t,d,1) = \max \{ R(t-1,d,0), \alpha_t + R(t-1,d-1,0) \}
$$
Finally, the value of the right-hand side is
\[
R(\alpha, (I_t),J )= \frac{ n - \lvert I_n \rvert }{n} R(n,|J|,0) + \frac{\lvert I_n \rvert}{n} R(n,|J|,1).
\]
The number of states is $ n \times \lvert J \rvert \times 2 =  O(n^2) $ and the number of operations to calculate a state's value is constant, so the entire recursion takes $ O(n^2) $ time.
\end{proof}

In the remainder of this section, we study particular cases of inequalities \eqref{eq:dp_ineq}. We construct them intuitively using probabilistic arguments, but their right-hand sides can also be calculated directly using the DP from Proposition \ref{prop:dp_ineq}.

As a first example, let $i\in N$, $j\in V$ and $t\in [n-1]$. Matching $i$ to $j$ in stage $t$ implies the intersection of two independent events. First, $j$ is not matched in any previous stage $[t+1,n]$, and second, $i$ appears in stage $t$. In terms of probability this means
$$
\PP(\text{match $i$ with $j$ in $t$})\leq \frac{1}{n}(1-\PP(\text{match $j$ in $[t+1,n]$})),
$$
which is equivalent to
$$
\PP(\text{match $j$ in stages $[t+1,n]$})+n\PP(\text{match $i$ with $j$ in $t$})\leq 1.
$$
The previous expression is equivalent to 
\begin{equation}\label{eq:ineq_Jsize_1}
\sum_{\tau=t+1}^n Z_{N,j}^{\tau}+nz_{i,j}^{t}\leq 1 \qquad \forall \ i\in N, \ j\in V, \ t\in [n].
\end{equation}
Inequality family \eqref{eq:ineq_Jsize_1} corresponds to a particular case of \eqref{eq:dp_ineq}, with $\lvert J \rvert = 1$, $I_\tau=N$ for $\tau\in [t+1,n]$, $\lvert I_t \rvert = 1$, $I_\tau=\emptyset$ for $\tau\leq t-1$, $\alpha_\tau=1$ for $\tau\in [t+1,n]$, $\alpha_{t}=n$, and $\alpha_\tau=0$ for $\tau\leq t-1$. Furthermore, for a fixed $ j \in V $ and $ t = 1 $, by summing the inequalities over all $ i \in N $ we obtain \eqref{eq:obm_relax_prev_jcap}.

\begin{proposition}\label{prop:ineq_Jsize_1}
Constraints \eqref{eq:ineq_Jsize_1} are facet-defining for the polyhedron of achievable probabilities $Q$ when $ t \leq n - 1 $.
\end{proposition}
\begin{proof}
Fix $i\in N, \ j\in V, \ t\in [n-1]$. We construct the following $n^3$ affinely independent points corresponding to policies that satisfy \eqref{eq:ineq_Jsize_1} with equality:
\begin{itemize}
\item[1.] Policy for $(i,j,t)$: if $i$ appears in stage $t$, then match it to $j$ with probability $1/n$. This corresponds to the point $\frac{1}{n}e^t_{i,j}.$
\item[2.] Policy for $(k,j,\tau)$ with any $k\in N$, and any $\tau\in[t+1,n]$: If $k$ appears in stage $\tau$, match it to $j$ with probability $1/n$, but if $k$ does not appear and  $i$ appears in stage $t$, we match $i$ to $j$ with probability $\frac{1}{n}\left(1-\frac{1}{n}\right)$. This corresponds to the point $\frac{1}{n}e_{k,j}^\tau + \frac{1}{n}\left(1-\frac{1}{n}\right)e^t_{i,j}.$ As we chose any $k$ and any $\tau$, we have $n(n-t)$ points.
\end{itemize}
So far, we only have $n(n-t)+1$ points. For the remaining points, we can use modifications of policy 1 above.
\begin{itemize}
\item Policy for $(k,j,\tau)$ with any $k\in N$ and $\tau\leq t-1$: If $i$ appears in stage $t$ with probability $1/n$, then match it with $j$; if $i$ does not appear (with probability $1-1/n$), and if $k$ appears in stage $\tau$ (with probability $1/n$), then match it with $j$. This corresponds to $\frac{1}{n}e^t_{i,j}+\frac{1}{n}\left(1-\frac{1}{n}\right)e^\tau_{k,j}$. As we chose any $k$, and any $\tau\leq t-1$, we have $n(t-1)$ points. 
\item Policy for $(k,\ell,\tau)$ with any $k\in V$, $\ell\in V$ such that $\ell\neq j$, and $\tau\in[n]$:  if $i$ appears in stage $t$ with probability $1/n$, then match it with $j$; if $k$ appears in stage $\tau$ (with probability $1/n$), then match it with $\ell$. This corresponds to $\frac{1}{n}e^t_{i,j}+\frac{1}{n}e^\tau_{k,\ell}$. In total, this yields $n(n-1)n$ points. 
\item Policy for $(k,j,t)$ with $k\in V$ such that $k\neq i$: if $i$ appears in stage $t$ with probability $1/n$, then match it with $j$; if $k$ appears in stage $t$ (with probability $1/n$), then match it with $j$. This corresponds to $\frac{1}{n}e^t_{i,j}+\frac{1}{n}e^t_{k,j}$. In this family, we have $n-1$ points. 
\end{itemize}
If we order these points in a suitable way, they form the columns of a block matrix
$$A=\begin{pmatrix} A_1 & A_2\\ 0 & A_3\end{pmatrix},$$
where $A_1$ is upper triangular and $A_3$ is a diagonal matrix. $A_1$ is formed by the first $n(n-t)+1$ points from policy 1 and the policies of item 2, while $ A_2 $ and $ A_3 $ are given by the remaining points. All diagonal entries of $A_1$ and $A_3$ are positive, implying that $A$ has positive determinant. This shows that the points previously described are linearly independent, completing the proof.
\end{proof}

We next compare the inequalities we have introduced so far to the known results for the lower-dimensional polyhedron $Q'$ of achievable probabilities that are not time-indexed, which we detail in Section \ref{sec:desc}. Recall that $Q'$ is a projection of $Q$ obtained by aggregating variables $z_{ij}^t$ over all stages, $z_{ij}=\sum_{t\in[n]}z_{ij}^t$. We already indicated how inequality families \eqref{eq:obm_relax_prev_icap} and \eqref{eq:obm_relax_prev_jcap} are respectively implied by \eqref{eq:prob_bound} and \eqref{eq:ineq_Jsize_1}. We next discuss the right-star inequalities \eqref{eq:right-star_prev}.

\begin{theorem}\label{prop:dominate_rightstar}
Inequalities \eqref{eq:ineq_Jsize_1} imply the right-star inequalities \eqref{eq:right-star_prev}.
\end{theorem}
\begin{proof}
Fix $ j \in V $ and $ I \subseteq N $.
First, for $ t = n $ \eqref{eq:ineq_Jsize_1} is simply $n z_{ij}^n \leq 1$ (it is also a weakened version of \eqref{eq:prob_bound}), and summing over  $I$ we get $n\sum_{i\in I}z_{ij}^n\leq |I|$. For $ t \leq n - 1 $, if we sum over $i\in I$ in \eqref{eq:ineq_Jsize_1} we get
$$|I|\sum_{\tau\in [t+1,n]}\sum_{k\in N}z_{k,j}^{\tau}+n\sum_{i\in I}z_{i,j}^{t}\leq |I|, \qquad \forall \ t\in[n-1],$$
and since $\sum_{i\in I}z_{i,j}^{\tau}\leq\sum_{k\in N}z_{k,j}^{\tau}$, we have
$$|I|\sum_{\tau\in [t+1,n]}\sum_{i\in I}z_{i,j}^{\tau}+n\sum_{i\in I}z_{i,j}^{t}\leq |I|, \qquad \forall \ t\in[n-1].$$
Then, multiply each inequality for $t\in[n-1]$ by $\frac{1}{n}\left(1-\frac{|I|}{n}\right)^{t-1}$, and add all of them (including the one for $t = n$); the resulting coefficient for each $ z_{ij}^t $ is
\[
\left(1-\frac{|I|}{n}\right)^{t-1} + \sum_{ \tau \leq t - 1 } \frac{\lvert I \rvert}{n}\left(1-\frac{|I|}{n}\right)^{\tau-1} = 1.
\]
We thus obtain $\sum_{t\in[n]}\sum_{i\in I}z_{ij}^t=\sum_{i\in I}z_{ij}$ in the left-hand side. In the right-hand side, we get
\[
\frac{|I|}{n}\sum_{t=1}^n\left(1-\frac{|I|}{n}\right)^{t-1}=1-\left(1-\frac{|I|}{n}\right)^n. \qedhere
\]
\end{proof}

This result shows that inequalities \eqref{eq:prob_bound} and \eqref{eq:ineq_Jsize_1} yield an upper bound that theoretically dominates the bound given by LP \eqref{eq:obm_relax_prev} with additional inequalities \eqref{eq:right-star_prev}, the best empirical bound previously known for OBM \cite{torrico_etal18}. In terms of dimension, the LP given by \eqref{eq:prob_bound} and \eqref{eq:ineq_Jsize_1} with non-negativity constraints has $ O(n^3) $ inequalities in $ \RR^{n^3} $, while \eqref{eq:obm_relax_prev} with \eqref{eq:right-star_prev} has exponentially many inequalities in $ \RR^{n^2} $.

\subsection{Policy Design} 

Theorem \ref{prop:dominate_rightstar} establishes that an LP in the space of $ z_{ij}^t $ variables with inequalities \eqref{eq:prob_bound} and \eqref{eq:ineq_Jsize_1}, 
\begin{align}\label{eq:prob+J1LP}
\max_{z \geq 0} \biggl\{ \sum_{i \in N} \sum_{j \in V} \sum_{t \in [n]} w_{ij}^t z_{ij}^t : \eqref{eq:prob_bound}, \eqref{eq:ineq_Jsize_1} \biggr\} ,
\end{align}
is guaranteed to provide a bound at least as good as the state of the art. 
%
We can also devise a policy from the LP \eqref{eq:prob+J1LP}, in a similar fashion to dynamic bid policies from network revenue management \cite{adelman07}. Denote by $\lambda_i^t\geq0$ and $\mu_{ij}^t\geq0$ the dual multipliers corresponding to constraints \eqref{eq:prob_bound} and \eqref{eq:ineq_Jsize_1} respectively. Along the lines of \cite{adelman07,torrico_etal18} and other approximate DP approaches, we construct an approximation of the true value function \eqref{eq:DP}: 
Interpret each $ \lambda_i^t $ as the value of having an impression of type $ i $ appear in period $ t $, and similarly interpret each $ \mu_{ij}^t $ as the value of having impression $ i $ appear in period $ t $ when ad $ j $ is available to match. For state $ (t, i, S) $ this yields the value function approximation
\begin{align}\label{eq:approx_function}
v_t(i,S)\approx \lambda_i^t+\sum_{\tau\in[t-1]}\EE_\eta [\lambda_{\eta}^\tau]+\sum_{j\in S} \biggl( \mu_{ij}^t+ \sum_{\tau\in [t-1]} \EE_\eta [\mu_{\eta j}^\tau] \biggr).
\end{align}
By imposing the constraints from \eqref{eq:obm_lp_val} on this approximation of the value function, we obtain the dual of \eqref{eq:prob+J1LP}:
\begin{align*}
\min_{ v \geq 0 } ~ & \EE_{\eta}[v_n(\eta,V)] & \min_{ \lambda, \mu \geq 0 } ~ & \sum_{ t \in [n] } \biggl( \EE_\eta[ \lambda_\eta^t ] + \sum_{ j \in V } \EE_\eta[ \mu_{\eta j}^t ] \biggr) \\
\text{s.t.\ } & v_{t}(i,S\cup j)-\EE_{\eta}[v_{t-1}(\eta,S)]\geq w^t_{ij}, \quad \xrightarrow{\eqref{eq:approx_function}} & \text{s.t.\ } & \lambda_i^t + \mu_{ij}^t + \sum_{\tau \in [t-1]} \EE_\eta[ \mu_{\eta j}^\tau ] \geq w_{ij}^t . \\ 
& v_{t}(i,S)-\EE_{\eta}[v_{t-1}(\eta,S)]\geq 0, 
\end{align*}
Furthermore, by replacing \eqref{eq:approx_function} in the DP recursion \eqref{eq:DP} for a state $ (t, i, S) $, we get the heuristic policy 
\begin{gather}
\argmax \Bigl\{ \max_{ j \in S } \{ w_{ij}^t + \EE_\eta[ v_{t-1}( \eta, S \setminus j ) ] \}, \EE_\eta[ v_{t-1}( \eta, S ) ] \Bigr\} \notag\\
\overset{\eqref{eq:approx_function}}{\approx} \argmax \Biggl\{ \max_{ j \in S } \biggl\{ w_{ij}^t + \sum_{\tau\in[t-1]} \biggl( \EE_\eta [\lambda_{\eta}^\tau] + \sum_{\ell \in S \setminus j} \EE_\eta [\mu_{\eta \ell}^\tau] \biggr) \biggr\}, \notag\\ 
\sum_{\tau\in[t-1]}\EE_\eta [\lambda_{\eta}^\tau] + \sum_{\ell \in S}\sum_{\tau\in [t-1]} \EE_\eta [\mu_{\eta \ell}^\tau] \Biggr\} \notag\\
= \argmax \Biggl\{
\max_{j\in S}\biggl\{w^t_{ij} - \sum_{ \tau \in [t-1] } \EE_\eta[ \mu_{\eta j}^\tau ] \biggr\},
0 \Biggr\}. \label{eq:time_index_policy}
\end{gather}
Intuitively, this policy evaluates the net benefit of a potential match of impression $ i $ to ad $ j $ in period $ t $ as the match's weight minus the value we give up by losing ad $ j $ in the subsequent remaining periods. The policy chooses the match with the largest such benefit (if positive), and otherwise discards the impression. 


\subsection{Polyhedral Study} 

Inequalities \eqref{eq:ineq_Jsize_1} correspond to a particular case of \eqref{eq:dp_ineq}, when the fixed set of ads $J$ has one element. We can apply a similar idea to a subset of any size; 
take the next simplest case of \eqref{eq:dp_ineq}, a set of size two, say $J=\{j_1,j_2\}$. Consider also two impressions $i_1, i_2\in N$, where we may have $ i_1 = i_2 $. In terms of probability, the event of matching $i_2$ with $j_1$ or $j_2$ in stage $t$ implies $i_2$ must appear in stage $t$ with probability $1/n$ and either of two events happens: First, neither $j_1$ nor $j_2$ are matched in stages $[t+2,n]$, and then $i_1$ appears in stage $t+1$ with probability $1/n$ (and can be matched to one of the ads or not); and second, $j_1$ or $j_2$ (but not both) are matched in stages $[t+2,n]$, and $i_1$ is not matched to $j_1$ nor $j_2$ in stage $t+1$ (this includes the case of another impression being matched to one of them). Since matching $j_1$ or $j_2$ in $ t $ are mutually exclusive events, we have the inequality
\begin{gather*}
\PP(\text{match $i_2$ with $j_1$ or $j_2$ in $t$})\leq \\ \frac{1}{n}\bigg[\frac{1}{n}\left(1-\PP(\text{match $j_1$ or $j_2$ in $[t+2,n]$})\right)+\left(1-\PP(\text{match $i_1$ with $j_1$ or $j_2$ in $t+1$})\right)\bigg] .
\end{gather*}
In terms of variables $z$, this is equivalent to
\begin{align}\label{eq:ineq_Jsize_2}
\begin{split}
& \sum_{\tau\in [t+2,n]} Z_{N,J}^{\tau}+n Z_{i_{t+1},J}^{t+1}+ n^2 Z_{i_t,J}^{t}\leq 1+n \\
& \qquad \qquad \qquad \forall \ i_t,i_{t+1}\in N, \ J\subseteq V, \ |J|=2, \ t\in [n-2].
\end{split}
\end{align}
This probabilistic argument can be generalized for any set $J\subseteq V$ with $|J| = h \in[n-1]$ and any $ t \leq n-h $. Let $(i_1,\ldots,i_h)$ be a sequence of nodes  in $N$ allowing repeats; 
the general constraint corresponds to
\begin{align*}
&\PP(\text{match $i_{h} $ to some $ j \in J $ in $t$})\\
&\leq \frac{1}{n}\left[\frac{1}{n^{h-1}}\left(1-\PP(\text{match $j_1$ or $j_2$ or $\ldots$ or $j_h$ in $[t+h,n]$})\right)\right.\\
&+\frac{1}{n^{h-2}}\left(1-\PP(\text{match $i_{1} $ to some $ j \in J $ in $t+h-1$})\right)\\
&+\frac{1}{n^{h-3}}\left(1-\PP(\text{match $i_{2} $ to some $ j \in J $ in $t+h-2$})\right)\\
&+\cdots+\left(1-\PP(\text{match $i_{h-1} $ to some $ j \in J $ in $t+1$})\right)\bigg].
\end{align*}
Therefore, we can give a general expression for this particular subclass of inequalities \eqref{eq:dp_ineq}:
\begin{align}\label{eq:ineq_Jsize_gen}
\begin{split}
\sum_{\tau=t+h}^{n} Z_{N,J}^{\tau}&+\sum_{\tau=t}^{t+h-1}n^{t+h-\tau} Z_{i_{\tau},J}^{\tau}\leq 1+\sum_{\tau=1}^{h-1}n^{\tau}, \\
& \forall  \ J\subseteq V, \ \lvert J \rvert = h\in[n-1], \ t\in [n-h], \ i_t,\dotsc,i_{t+h-1}\in N.
\end{split}
\end{align}
\begin{theorem}
\label{theo:ineq_Jsize_gen}
Constraints \eqref{eq:ineq_Jsize_gen} are facet-defining for $Q$.
\end{theorem}
The proof of this theorem can be found in the Appendix. 

So far we have only considered either $I_\tau=N$ or $ \lvert I_\tau \rvert = 1$ within inequalities \eqref{eq:dp_ineq}. We next propose a generalization for other sets $I$.
Consider the case $J = \{j\}$, and any subset $I\subseteq N$; suppose we naively apply the same argument behind inequality \eqref{eq:ineq_Jsize_1}. Matching an element of $I$ with $j$ in stage $t$ implies the intersection of two independent events: First, $j$ is not matched in stages $[t+1,n]$, and second, some element in $I$ appears in stage $t$. In probabilistic terms,
$$
\PP(\text{match any element in $I$ with $j$ in $t$})\leq \frac{|I|}{n}(1-\PP(\text{match $j$ in $[t+1,n]$})),
$$
which is equivalent to
$$
|I|\sum_{\tau=t+1}^n Z_{N,j}^{\tau}+n Z_{I,j}^{t}\leq |I|.
$$
However, this inequality is made redundant by \eqref{eq:ineq_Jsize_1}, because we can sum over $i\in I$ for the same fixed $ t $ to get it.

Consider instead $J=\{j_1,j_2\}$, any $I_1 \subseteq N$ with $ \lvert I_1 \rvert \geq 2 $, and another impression $i_2\in N$; we apply the same argument used for \eqref{eq:ineq_Jsize_2}, but substituting $ I_1 $ for the single impression $ i_1 $. Matching $i_2$ with $j_1$ or $j_2$ in stage $t$ implies $i_2$ appears in stage $t$ with probability $1/n$, and either of two previous events happens: First, neither $j_1$ nor $j_2$ are matched in stages $[t+2,n]$, and then any element in $I_1$ appears in stage $t+1$ with probability $|I_1|/n$ (and is matched to one of the ads or not); and second, one of $j_1$ or $j_2$ is matched in stages $[t+2,n]$, and no element from $I_1$ is matched to $j_1$ nor $j_2$ in $t+1$ (this includes the case of another impression being matched to one of them). Since matching $j_1$ or $j_2$ in $ t $ are mutually exclusive, we have 
\begin{align*}
\PP(\text{match $i_2$ with $j_1$ or $j_2$ in $t$})&\leq \frac{1}{n}\bigg[\frac{|I_1|}{n}\left(1-\PP(\text{match $j_1$ or $j_2$ in $[t+2,n]$})\right)\\
&+\left(1-\PP(\text{match some $i \in I_1$ with $j_1$ or $j_2$ in $t+1$})\right)\bigg],
\end{align*}
which is equivalent to
\begin{equation}\label{eq:ineq_anyI_Jsize2}
\lvert I_1 \rvert \sum_{\tau=t+2}^n Z_{N,J}^{\tau}+n Z_{I_1,J}^{t+1}+n^2 Z_{i_2,J}^{t}\leq \lvert I_1 \rvert +n .
\end{equation}
As with inequalities \eqref{eq:ineq_Jsize_1}, if we attempt to naively extend this argument by considering a larger set $ I_2 $ instead of the single impression $ i_2 $, we simply get redundant inequalities. However, we can generalize
\eqref{eq:ineq_anyI_Jsize2} using the same argument for \eqref{eq:ineq_Jsize_gen}: For any $ I \subseteq N $ with $ \lvert I \rvert = r \leq n-1 $ and any $ J \subseteq V $ with $ \lvert J \rvert = h \leq n -1 $, we obtain the inequalities
\begin{align}\label{eq:ineq_IJ_gen}
\begin{split}
& r \sum_{\tau=t+h}^{n} Z_{N,J}^{\tau}+n Z_{I,J}^{t+h-1}+\sum_{\tau=t}^{t+h-2}n^{t+h-\tau} Z_{i_{\tau},J}^{\tau}\leq r+\sum_{\tau=1}^{h-1}n^{\tau}, \\
& \forall \   J\subseteq V,  \lvert J \rvert = h\in[n-1], \ I\subseteq N, \lvert I \rvert = r\in [n-1], t\in [n-h], \ i_t,\dotsc,i_{t+h-2}\in N.
\end{split}
\end{align}

\begin{theorem}\label{theo:ineq_IJ_gen}
Constraints \eqref{eq:ineq_IJ_gen} are facet-defining for $Q$.
\end{theorem}

For a proof of this theorem, see the Appendix.

In inequalities \eqref{eq:ineq_IJ_gen}, we do not consider $ I = N $. Suppose we apply the same argument for \eqref{eq:ineq_anyI_Jsize2} in this case; we then obtain
$$n\sum_{\tau=t+h}^{n} Z_{N,J}^{\tau}+n Z_{N,J}^{t+h-1}+\sum_{\tau=t}^{t+h-2}n^{t+h-\tau} Z_{i_{\tau},J}^{\tau}\leq n+\sum_{\tau=1}^{h-1}n^{\tau}.$$
Dividing by $n$, we get
$$\sum_{\tau=t+h}^{n} Z_{N,J}^{\tau}+ Z_{N,J}^{t+h-1}+\sum_{\tau=t}^{t+h-2}n^{t+h-\tau-1} Z_{i_{\tau},J}^{\tau}\leq 1+\sum_{\tau=1}^{h-1}n^{\tau-1},$$
which is equivalent to
$$\sum_{\tau=t+h}^{n} Z_{N,J}^{\tau}+ Z_{N,J}^{t+h-1}+n Z_{i_{t+h-2},J}^{t+h-2}+\sum_{\tau=t}^{t+h-3}n^{t+h-\tau-1} Z_{i_{\tau},J}^{\tau}\leq 2+\sum_{\tau=1}^{h-2}n^{\tau}.$$
This idea also generates valid inequalities for $Q$, but we can generalize it even more. Thus far, we consider an arbitrary subset $I$ in stage $t+h-1$, but in this last inequality the arbitrary subset can ``shift'' to stage $t+h-2$, so we can actually state a more general valid inequality
\begin{align*}
r\sum_{\tau=t+h-1}^{n}& Z_{N,J}^{\tau}+n Z_{I,J}^{t+h-2}+\sum_{\tau=t}^{t+h-2}n^{t+h-\tau-1} Z_{i_{\tau},J}^{\tau}\leq 2r+\sum_{\tau=1}^{h-2}n^{\tau}, \\
 & \forall \  t\in [n-h], \ J\subseteq V, \ l\in[n-1], \ I\subseteq N, \ r\in [n-1], \ i_t,\dotsc,i_{t+h-2}\in N.
\end{align*}
In this last inequality we only consider $r\in[n-1]$, but as before, we can actually again take $I=N$ in stage $t+h-2$. After dividing by $n$, we get
$$
\sum_{\tau=t+h-1}^{n} Z_{N,J}^{\tau}+ Z_{N,J}^{t+h-2}+\sum_{\tau=t}^{t+h-2}n^{t+h-\tau-2} Z_{i_{\tau},J}^{\tau}\leq 2+\sum_{\tau=1}^{h-2}n^{\tau-1},
$$
equivalent to
$$
\sum_{\tau=t+h-2}^{n} Z_{N,J}^{\tau}+n Z_{i_{t+h-3},J}^{t+h-3}+\sum_{\tau=t}^{t+h-4}n^{t+h-\tau-2} Z_{i_{\tau},J}^{\tau}\leq 3+\sum_{\tau=1}^{h-3}n^{\tau} .
$$
This is also a valid inequality for $Q$, and we can continue doing this process as many as $h-2$ times until we get
$$
r\sum_{\tau=t+2}^{n} Z_{N,J}^{\tau}+n Z^{t+1}_{i_{t+1},J}+n^2 Z_{i_{t},J}^{t}\leq r(h-1)+n,
$$
which is a generalization of \eqref{eq:ineq_Jsize_2}. Finally, if we do this process one more time we get
$$
\sum_{\tau=t+1}^{n} Z_{N,J}^{\tau}+n Z_{i_{t},J}^{t}\leq h,
$$
which is clearly implied by summing over $j \in J$ in \eqref{eq:ineq_Jsize_1}. 

Denote by $q\in[0,h-2]$ the number of times we apply this procedure. We now state the most general family of valid inequalities we have obtained as a specific subclass of \eqref{eq:dp_ineq}. We subdivide this class using a 4-tuple $(h,r,t,q)$, which respectively identifies the size of $J$, the size of $I$, the stage, and the number of times we apply the previous procedure. So, for any $J\subseteq V$ with $\lvert J \rvert = h\in[2,n-1]$, any $I\subseteq N$ with $ \lvert I \rvert = r\in[n-1]$, any $t\in[n-h]$, and any $q\in[0,h-2]$, we have the following valid inequality
\begin{align}\label{eq:ineq_general}
r \sum_{\tau = t+h-q}^n Z_{N,J}^{\tau} +n Z^{t+h-q-1}_{I,J}&+\sum_{\tau=t}^{t+h-q-2}n^{t+h-q-\tau} Z^{\tau}_{i_{\tau},J} \leq r(q+1)+\sum_{\tau=1}^{h-q-1}n^{h-q-\tau} .
\end{align}
We have already proved that the inequalities given by $(h,r,t,0)$ are facet-defining; here we give the general result.
\begin{theorem}
\label{theo:ineq_general}
Inequalities \eqref{eq:ineq_general} identified by $(h,r,t,q)$ are facet-defining for $Q$ when $ h \in [2, n-1] $, $ r \in [n-1] $, $ t \in [n - h] $ and $ q \in [0, h-2] $.
\end{theorem}

Finally, we show the following complexity result for this general family of facet-defining inequalities.
\begin{proposition}\label{prop:NPhard}
It is NP-hard to separate inequalities \eqref{eq:ineq_IJ_gen}, and thus also \eqref{eq:ineq_general}.
\end{proposition}
\begin{proof}
Fix $ h = r $ and $ t $.
Suppose we have a solution $ z $ that is zero (or constant) in all values except for stage $ t + h - 1 $. In this case, the separation problem for this $ h $, $ r $ and $ t $ is equivalent to
\[
\max
\{ Z_{I,J}^{t+h-1} : I \subseteq N, J \subseteq V, \lvert I \rvert = \lvert J \rvert = h = r \}.
\]
This is a weighted version of the maximum balanced biclique problem, which is NP-hard \cite{dawande_etal96}. For $ h \neq r $, the problem can be transformed to make the two cardinalities equal.
\end{proof}
%

\section{Computational Study}\label{sec:exp}


\subsection{Description of Experiments}

Our main experimental goal is testing the effectiveness of our new dynamic relaxations and comparing the new bounds given by these relaxations to several benchmarks. As a secondary goal, we also study the heuristic policy \eqref{eq:time_index_policy} implied by our relaxation and compare it with the best empirically performing policy from the literature.

The best empirical bound previously known for OBM is the LP \eqref{eq:obm_relax_prev} with additional inequalities \eqref{eq:right-star_prev} \cite{torrico_etal18}. Our results in the previous section establish that \eqref{eq:prob+J1LP} is guaranteed to be no worse.
So we compare these two bounds to determine how much of an improvement the latter LP \eqref{eq:prob+J1LP} offers over the former. In addition, we would like to examine if some of the other inequalities we introduce can further improve the bound. However, testing these additional inequality classes involves computational challenges. In particular, the LP's dimension grows as $ n^3 $, implying a relatively large number of variables even for moderately sized instances. This practically limits both the number of inequalities we consider, and the actual number we can dynamically add to the LP. To this end, we test adding inequalities \eqref{eq:ineq_Jsize_2} to \eqref{eq:prob+J1LP}; these inequalities are still polynomially many, $ \Theta(n^5) $, and relatively efficient to separate over. We also considered including inequalities \eqref{eq:validExt2-n-1}, that is, the special case of \eqref{eq:ineq_Jsize_gen} with $ h = n - 1 $ and $ t = 1 $, as they are also simple to separate over despite numbering $ \Theta(n^n) $. However, our preliminary experiments revealed numerical difficulties with these inequalities; the smallest non-zero coefficient is $ 1 $, while the largest is $ n^{n-1} $, and although these numbers (and all of the coefficients and right-hand sides of our inequalities) require $ O(n \log n) $ space in binary representation and are thus of polynomial size, in practical terms these differences in scale make it difficult to even determine whether a particular inequality is violated, and thus to separate over the entire family. We therefore did not include these inequalities in our experiments.

As for lower bounds given by heuristic policies, \cite{torrico_etal18} introduce a \emph{time-dependent ranking} policy derived from \eqref{eq:obm_relax_prev} with additional inequalities \eqref{eq:right-star_prev}, and results in this paper establish it as the best performing policy among several from the literature. We use it as a benchmark to test policy \eqref{eq:time_index_policy}. 

Finally, we include as additional benchmarks the optimal value given by the DP recursion \eqref{eq:DP} (for small instances where it can be computed), as well as the max-weight expected off-line matching, the expected value of the matching we would choose if we could observe the entire sequence of realized impressions before making a decision. This latter benchmark is also an upper bound on the optimal value, as it relaxes non-anticipativity.

\subsection{Instance Design and Implementation}

All of the instances we tested have $ n = m = T $, with binary edge weights constant over time, $ w_{ij}^t = w_{ij} \in \{0, 1\} $. In other words, all the instances are max-cardinality OBM problems with static edges; the static weights are required because the benchmarks we use to compare against do not accommodate weights that vary over time. We generate  instances with the following rubrics:
\begin{enumerate}
\item $ 20 $ \emph{small} instances with $n=10$, each one randomly generated by having a possible edge in $ N \times V $ be present independently with a probability of $ 25\% $, so the expected average degree is 2.5.
\item $ 20 $ \emph{large, dense} instances with $n=100$, each one randomly generated by having a possible edge in $N\times V$ be present independently with a probability of $ 10\% $, so the expected average degree is 10. 
\item $ 20 $ \emph{large, sparse} instances with $n=100$, each one randomly generated by having a possible edge in $N\times V$ be present independently with probability of 2.5\%, so the expected average degree is 2.5. 
\item A set of \emph{large, $k$-regular} graphs with $ n = 100 $ and $ k \in \{ 3, 4, 5, 6 \} $, constructed in the following way: Indexing both impressions and ads from $0$ to $n-1$, each impression $ i $ is adjacent to ads $ \{i,i+1,\ldots,i+k-1\} \mod k $. The motivation for this last set of experiments is that the relaxations and policies may behave differently on instances with a high degree of symmetry, as opposed to randomly generated instances.
\end{enumerate}
For any experiment requiring simulation, including computing the expected value of the heuristic policies and the max-weight off-line matching, we used $ 20,000 $ simulations and report the sample mean and sample standard deviation.

For small instances, all the bound experiments took a few seconds on average. For the larger instances, we solved the benchmark LP's following the approach from \cite{torrico_etal18}. For the new bounds, we formulated \eqref{eq:prob+J1LP} but eliminated all variables corresponding to missing edges; this results in models with an average of 25,000, 100,000 and $ 10,000 \times k $ variables for sparse, dense and regular instances respectively. The solution times for these LP's were roughly one hour for dense instances, and under a minute for sparse instances, with regular instances varying as $k$ grows.
%
After solving this LP, we switched to constraint generation for inequalities \eqref{eq:ineq_Jsize_2}; however, after preliminary experiments we did this only for small and large sparse instances, because of protracted solve times with minimal bound improvement in the other cases.


\subsection{Summary of Results}


Table \ref{table:ratios1} summarizes the experiment results for all instances except regular ones, which are detailed individually below. For each instance class, in each row we present the geometric mean of each bound or policy's ratio to a fixed benchmark -- the DP value for small instances and the max-weight expected off-line matching for large ones. 
We also report the sample standard deviation of these ratios in parenthesis.
\begin{table}[h]
\begin{center}
{\small
{\begin{tabular}{|c|c|c|c|}
\hline
Bound/Policy                 &  Small   & Large Dense & Large Sparse \\ \hline
\eqref{eq:obm_relax_prev} + \eqref{eq:right-star_prev}          &  1.0845 (0.0263)               & 1.0004 (0.0004)       & 1.0813 (0.0085)     \\ \hline
\eqref{eq:prob+J1LP}        &  1.0407  (0.0099)           & 0.9739 (0.0011)   & 1.0283  (0.0050)    \\ \hline
\eqref{eq:prob+J1LP} + \eqref{eq:ineq_Jsize_2} &  1.0381 (0.0091)       &   -  & 1.0278 (0.0050)       \\ \hline
 \hline
Off-Line Exp.\ Matching          &  1.0253 (0.0134)         & 1          & 1            \\ \hline
\eqref{eq:DP}                                    &  1           	                      & -           & -            \\ \hline
\hline
Policy \eqref{eq:time_index_policy} 	 & 0.9974  (0.0017)           & 0.9561 (0.0025)   & 0.9627 (0.0066)    \\ \hline
TD Ranking Policy \cite{torrico_etal18}   		 & 0.9901  (0.0080)           & 0.9539 (0.0025)   & 0.9529 (0.0073)    \\ \hline
\end{tabular}}
\caption{Summary of experiment results.}\label{table:ratios1}}
\end{center}
\end{table}

We know from our results in the previous section that the bound given by \eqref{eq:prob+J1LP} is guaranteed to outperform the bound given by \eqref{eq:obm_relax_prev} with \eqref{eq:right-star_prev}. However, our results show that the improvement is significant, with the new bound cutting the gap by about $ 4\% $ on average for small instances and approximately $3\%$ to $ 5\% $ for large ones. Furthermore, the improvement is consistent across all the tested instances; in particular, the two bounds never match.


The results for large, dense instances are particularly noteworthy; the new bound from \eqref{eq:prob+J1LP} also beats the max-weight expected off-line matching, not only on average but in every instance. Our intuition for this result is the following. In dense instances, there is likely a perfect or near-perfect matching in every realization, and thus the off-line matching will be very close to $ n $ in expectation. Of course, even in a dense instance it may be that no online policy can guarantee a perfect or near-perfect matching, and explicitly accounting for temporal aspects of the problem, particularly as inequalities \eqref{eq:ineq_Jsize_1} do, captures this phenomenon and tightens the bound, unlike the off-line matching or the more static approach of the benchmark LP.


Interestingly, our results also reveal that the bound from \eqref{eq:prob+J1LP} is not improved much with the addition of inequalities \eqref{eq:ineq_Jsize_2}, especially considering the significant additional computing time. 
In light of these results, we also performed experiments to test the bound given by \eqref{eq:prob_bound} and \eqref{eq:ineq_Jsize_2} only (without inequalities \eqref{eq:ineq_Jsize_1}). However, the resulting bounds were much looser, confirming that inequalities \eqref{eq:ineq_Jsize_1} are crucial to providing a tight bound.



In terms of policies, our new heuristic \eqref{eq:time_index_policy} is consistently better than the time-dependent ranking policy, the best performing policy from the literature. This improvement occurs in almost every tested instance, though the magnitude of the improvement varies. The new policy is near-optimal for small instances, and cuts the gap for large instances, by about $0.7\%$ to $1\%$ on average in absolute terms. This improvement in policy quality mirrors results in other areas, such as revenue management, where heuristic policies derived from time-indexed relaxations also outperform policies stemming from ``static'' LP's; see e.g.\ \cite{adelman07,zhang_adelman09}.

The results for regular graphs are in Table \ref{tab:cycles}, shown here in absolute terms since we are not averaging multiple experiments. We observe similar improvements in terms of upper bounds, where our new bound significantly cuts the gap, by around $7\%$. On the other hand, we observe no improvement on the policy side. Intuitively, this last result is unsurprising, since both heuristic policies depend on dual multipliers of LP's that are symmetric for regular instances, in the sense that they both have dual optimal solutions in which every value at some stage is equal. Both policies are thus choosing a match uniformly at random.
\begin{table}[h]
\begin{center}
{\small
 {\begin{tabular}{|c|c|c|c|c|c|}
\hline
Instance   & \eqref{eq:obm_relax_prev} + \eqref{eq:right-star_prev} & \eqref{eq:prob+J1LP}  &   Exp.\ Matching & 
 \eqref{eq:time_index_policy} & TDR Policy \cite{torrico_etal18} \\ \hline
3-regular & 95.2447 & 87.9224 & 85.5680 & 79.8960 & 79.8960 \\ \hline
4-regular & 98.3130 & 90.9901 & 89.2247 & 82.8647 & 82.8647 \\ \hline
5-regular & 99.4079 & 92.8303 & 91.5918 & 85.1153 & 85.1153 \\ \hline
6-regular & 99.7945 & 94.0548 & 93.2837 & 86.9821 & 86.9821 \\ \hline
\end{tabular}}
 \caption{Experiment results for regular graphs.}\label{tab:cycles}}
\end{center}
\end{table}


\section{Conclusions}\label{sec:conc}

This work proposes dynamic relaxations for the i.i.d.\ OBM problem and studies them from a polyhedral point of view. While several past results have used different LP relaxations, ours is the first to explicitly consider the time dimension. Among various benefits of the approach, this allows for the model to accommodate time-varying edge weights, and also allows us to elide the instance's structure in the analysis, by capturing all of this information in the problem's objective. Our study centers on the polyhedron of time-indexed achievable probabilities $Q$, and includes a large class of facet-defining inequalities for this polytope based on choosing complete bipartite subgraphs. Furthermore, our experiments confirm that the time-indexed approach offers significant benefits; the bound given by the simplest members of our proposed inequality family already significantly outperforms the best empirical bounds given by static LP's, and a heuristic policy derived from this new bound also significantly outperforms the best policy based on a static relaxation.


Our results motivate a variety of questions for future work. For example, we would like to understand the structure of valid inequalities that are not based on complete bipartite subgraphs, to potentially further improve the dual bound. Using Fourier-Motzkin elimination and the software PORTA, we have derived the full description of $Q$ for small cases, such as $ n = m = T = 3 $. We observed many different inequalities, including some that are somewhat similar to our general family \eqref{eq:dp_ineq}, so there may be a more general class to propose that still lends itself to analysis similar to ours. 

Much of the literature on OBM studies the worst-case performance of heuristic policies based on relaxations. Although that was not our goal in this work, the positive empirical results we observed when implementing our new heuristic policy suggest a similar analysis for that policy, especially since it appears to differ in structural terms from many OBM heuristics.
More generally, an interesting question is whether a polyhedral analysis similar to ours can be applied to derive new bounds and policies in related online matching and resource allocation contexts.



\section*{Acknowledgments}
The authors' work was partially supported by the National Science Foundation under grant CMMI 1552479.

\bibliographystyle{amsplain}
\bibliography{adp_combopt}

\section{Appendix}

\subsection{Remaining Proofs}

\begin{proof}[Proof of Theorem \ref{theo:ineq_Jsize_gen}.]
The case $h=1$ is already covered by the proof of Proposition \ref{prop:ineq_Jsize_1}. Consider the case $h=n-1$ and $t=1$; the other cases follow a similar construction of linearly independent points. Let $J=\{0,\ldots, n-2\}$, and assume without loss of generality that $i_\tau=i$ for all $\tau\in[n-1]$. The specific inequality is
\begin{equation}\label{eq:validExt2-n-1}
Z_{N,J}^n+\sum_{\tau=1}^{n-1} n^{n-\tau} Z_{i,J}^\tau \leq 1+\sum_{\tau=1}^{n-2} n^{\tau}.
\end{equation}
We know that $z\in[0,1]^{n^3}$, but for the description of the points (and the proof) we will just consider the coordinates involved in the inequality, i.e., $z\in[0,1]^{p}$, where $p:=(2n-1)(n-1)$. For the rest of the points, the construction is similar to the one in the proof of Proposition \ref{prop:ineq_Jsize_1}. Recall that $e_{k,j}^{\tau}$ denotes the canonical vector in $[0,1]^{p}$, i.e. a vector with a 1 in coordinate $ (k, j, \tau ) $ and zero elsewhere, indicating a match of impression $k$ with ad $j$ in stage $\tau$. Consider the elements of $J$ as an $(n-1)$-tuple, i.e., $(0,\ldots,n-2)$. For $j\in J$, we define
\[
j+(0,\ldots,n-2):= (j,\ldots,j+n-2) \mod  (n-1).
\]
Any addition or substraction with $j\in J$ is modulo $(n-1)$ for the remainder of the proof. We denote the circulation of $J$ as the following set of $(n-1)$-tuples:
\begin{align*}
\circu(J)&:=\{j+(0,\ldots,n-2)\}_{j\in J}\\
&=\{(0,\ldots,n-2),(1,\ldots,n-2,0),\ldots,(n-2,0,\ldots,n-3)\}.
\end{align*}
Note that $\circu(J)$ can be viewed as a matrix. Each element of $\circu(J)$ corresponds to a sequence of ads in the process from stage $n$ to stage 1. Since we have $n$ stages and any of those sequences has size $n-1$, then clearly there is no matching in some stage or an element repeats. We now describe the family of linearly independent points.
\begin{enumerate}
\item[I.] Fix $k\in N$ and $j\in J$. In stage $n$, if node $k$ appears, then match it to node $j$, with probability $1/n$. For the remaining stages match according to  $(j,j+1,\ldots,j+n-2) \in \circu(J)$. In terms of probability, if $i$ appears in stage $n-1$, then it is matched to $j$ with probability $(1-1/n)\cdot 1/n$. For the rest, the probability is $1/n$. So, we have the point
 \begin{equation}\label{eq:family1}\frac{1}{n}e_{k,j}^n+\frac{1}{n}\left(1-\frac{1}{n}\right)e_{i,j}^{n-1}+\frac{1}{n}\sum_{\tau=1}^{n-2}e_{i,j+\tau}^{n-1-\tau}.\end{equation}
By a simple calculation, it is easy to see that each of these points achieves the righ-hand side of \eqref{eq:validExt2-n-1}. Since we chose an arbitrary $k\in N$ and $j\in J$, we have $n(n-1)$ points in this family.

\item[II.] Fix $j\in J$. In this family we repeat $ j $ in stages $n-1$ and $n-2$. If $i$ appears in stage $n-1$, match it to node $j$ with probability $1/n$. If $i$ appears in stage $n-2$ and it did not appear in $n-1$, match it to $j$ with probability $(1-1/n)\cdot 1/n$. For the remaining stages match according to $(j+n-2,j,j+1,\ldots,j+n-3) \in \circu(J)$; in particular, in stage $n$ match any $k\in N$ that appears with node $j+n-2$, in stage $ n-3 $ match $i$ to $j+1$ if it appears, and so forth. So, we have the point
\begin{equation}\label{eq:family2}
\frac{1}{n}\sum_{k\in N}e_{k,j+n-2}^n+\frac{1}{n}e^{n-1}_{i,j}+\frac{1}{n}\left(1-\frac{1}{n}\right)e_{i,j}^{n-2}+\frac{1}{n}\sum_{\tau=1}^{n-3}e_{i,j+\tau}^{n-\tau-2}
\end{equation}
By a simple calculation, we get the right-hand side of \eqref{eq:validExt2-n-1}. Since we chose an arbitrary $j\in J$, we have $n-1$ points in this family.

\item[III.] Fix $j\in J$; in this family we have two different options in stage $n-3$. 
If $i$ appears in stage $n-1$, match it to $j$ with probability $1/n$. If $i$ appears in stage $n-2$, match it to $j+1$, also with probability $1/n$. If $i$ appears in stage $n-3$, match it to $ j+1 $ with probability $(1-1/n)\cdot 1/n$, or if $ j+1 $ was matched in stage $n-2$, then to node $j$ with probability $(1-1/n)\cdot 1/n^2$.
For the remaining stages match according to $(j+n-2,j,j+1,\ldots,j+n-3)$; in stage $n$ match any $k\in N$ that appears to $j+n-2$, in stage $ n - 4 $ match $i$ to $ j+2 $ if it appears, and so forth. So, we have the point
\begin{align}\label{eq:family3}\frac{1}{n}\sum_{k\in N}e_{k,j+n-2}^n&+\frac{1}{n}e^{n-1}_{i,j}+\frac{1}{n}e_{i,j+1}^{n-2} \notag \\
&+\left(1-\frac{1}{n}\right)\left[\frac{1}{n^2}e_{i,j}^{n-3}+\frac{1}{n}e_{i,j+1}^{n-3}\right]+\frac{1}{n}\sum_{\tau=1}^{n-4}e_{i,j+\tau+1}^{n-\tau-3}
\end{align}
By a simple calculation, we get the right-hand side of \eqref{eq:validExt2-n-1}. Since we chose an arbitrary $j\in J$, we have $n-1$ points in this family.

\item[IV.] Fix $j\in J$ and stage $s\in [n-4]$; the previous family can be generalized for stage $s$, but increasing the number of options, i.e., in stage $s$ we have $n-s-1$ options from the previous stages.  
    If $i$ appears in stage $n-1$, match it to $j$ with probability $1/n$, if $i$ appears in stage $n-2$,  match it to $j+1$ with probability $1/n$, and  
    continue in this way until stage $s+1$, where if $i$ appears, match it to node $j+n-s-2$ with probability $1/n$. If $i$ appears in stage $s$, we consider ads $ ( j+n-s-2, \dotsc, j+1, j ) $ in this order of priority, so that $i$ is matched to $j+n-s-2$ with probability $(1-1/n)\cdot 1/n$; each subsequent ad's probability of being matched to $i$ decreases exponentially until $j$, which has probability $(1-1/n)\cdot 1/n^{n-s-1}$.
     For the remaining stages (including stage $n$) match according to 
     $(j+n-2,j,j+1,\ldots,j+n-3)$; in stage $n$ match any $k\in N$ that appears with $j+n-2$, in $s-1$ match $i$ to $ j+ n - s - 1 $ if it appears, etc. So, we have the point
\begin{align}\label{eq:family4}
\frac{1}{n}\sum_{k\in N}&e_{k,j+n-2}^n+\frac{1}{n}\sum_{\tau=0}^{n-s-2}e^{n-\tau-1}_{i,j+\tau}\notag \\
&+\left(1-\frac{1}{n}\right)\left[\sum_{\tau=0}^{n-s-2}\frac{1}{n^{n-\tau-s-1}}e^{s}_{i,j+\tau}\right]+\frac{1}{n}\sum_{\tau=1}^{s-1}e_{i,j+n-s-2+\tau}^{s-\tau}
\end{align}
The left-hand side of \eqref{eq:validExt2-n-1} evaluated at this point is
$$1+\sum_{\tau=s+1}^{n-1} \frac{n^{n-\tau}}{n}+\left(1-\frac{1}{n}\right)\sum_{\tau=0}^{n-s-2}\frac{n^{n-s}}{n^{n-\tau-s-1}}+\sum_{\tau=1}^{s-1}\frac{n^{n-\tau}}{n} =1+\sum_{\tau=1}^{n-2} n^{\tau}.$$
Finally, since we chose an arbitrary $j\in J$ and $s\in[n-4]$, we have $(n-1)(n-4)$ points in this family.

\item[V.]  Fix $j\in J$. For this family we do not match in stage $n$, and in the remaining stages we match according to $(j,j+1,\ldots,j+n-2) \in \circu(J)$. If $i$ appears in stage $n-1$ match it to $j$ with probability $1/n$, if $i$ appears in stage $n-2$,  match it to $j+1$, and so on. So we have the point
\begin{equation}\label{eq:family5}
\frac{1}{n}\sum_{\tau=0}^{n-2}e_{i,j+\tau}^{n-\tau-1}
\end{equation}
By a simple calculation, we get the right-hand side of \eqref{eq:validExt2-n-1}. Finally, since we chose an arbitrary $j\in J$, then we have $n-1$ points in this family.
\end{enumerate}
With these families, we have $p$ points in total. Denote by $(k,j,\tau)$ the index of a vector $z\in[0,1]^p$, which indicates that  $k\in N$ is matched to $j\in J$ in stage $\tau$. In any of these points consider the following order of components (starting from the first one): $(1,0,n)$, $(1,1,n)$, $\ldots$, $(1,n-2,n)$, $\ldots$, $(n,0,n)$, $\ldots$, $(n,n-2,n)$, $(i,0,n-1)$, $\ldots$, $(i,n-2,n-1)$, $\ldots$, $(i,0,1)$, $\ldots$, $(i,n-2,1)$.

The rest of the proof consists of showing that these families define a set of linearly independent points, and we prove this using Gaussian elimination. Arrange these points as column vectors in a matrix $A$,
$$
A=[\text{I} , \text{II}, \text{III}, \text{IV}, \text{V}]
=\begin{pmatrix}B_1 & B_2 \\ B_3 & B_4\end{pmatrix},
$$
where $B_1$ is a $n(n-1)\times n(n-1)$ diagonal matrix with entries $1/n$. These columns can be used to make $B_2$ a zero matrix, yielding
$$
\bar{A}=\begin{pmatrix}B_1 & 0 \\ B_3 & C\end{pmatrix}.
$$
Consider how the columns from families II, III, and IV look like after this elimination procedure (family V is not affected).
Fix $g\in J$ and sum every point \eqref{eq:family1} over $k\in N$; this yields
\begin{equation}\label{eq:family1a}
\frac{1}{n}\sum_{k\in N}e_{k,g}^n+\left(1-\frac{1}{n}\right)e_{i,g}^{n-1}+\sum_{\tau=1}^{n-2}e_{i,g+\tau}^{n-1-\tau} .
\end{equation}

\begin{enumerate}
\item[II$^a$.] Pick the point \eqref{eq:family2} associated with $g+1\in J$,
 \begin{equation}\label{eq:family2a}
\frac{1}{n}\sum_{k\in N}e_{k,g}^n+\frac{1}{n}e^{n-1}_{i,g+1}+\frac{1}{n}\left(1-\frac{1}{n}\right)e_{i,g+1}^{n-2}+\frac{1}{n}\sum_{\tau=1}^{n-3}e_{i,g+1+\tau}^{n-\tau-2} .
\end{equation}
Subtract \eqref{eq:family1a} from \eqref{eq:family2a} to get
$$
\frac{1}{n}e^{n-1}_{i,g+1}+\frac{1}{n}\left(1-\frac{1}{n}\right)e_{i,g+1}^{n-2}+\frac{1}{n}\sum_{\tau=1}^{n-3}e_{i,g+1+\tau}^{n-\tau-2}-\left(1-\frac{1}{n}\right)e_{i,g}^{n-1}-\sum_{\tau=1}^{n-2}e_{i,g+\tau}^{n-1-\tau} ,
$$
which is equivalent to
\begin{equation}\label{eq:family2b}
\frac{1}{n} e^{n-1}_{i,g+1} + \left(-1+\frac{1}{n}\right)e_{i,g}^{n-1}+\left(\frac{1}{n}-\frac{1}{n^2}-1\right)e_{i,g+1}^{n-2}+\left(-1+\frac{1}{n}\right)\sum_{\tau=2}^{n-2}e_{i,g+\tau}^{n-1-\tau} .
\end{equation}

\item[III$^a$.] Pick the point \eqref{eq:family3} associated with $g+1\in J$,
\begin{align}\label{eq:family3a}
\frac{1}{n}\sum_{k\in N}e_{k,g}^n&+\frac{1}{n}e^{n-1}_{i,g+1}+\frac{1}{n}e_{i,g+2}^{n-2}  \notag \\
&
+\left(1-\frac{1}{n}\right)\left[\frac{1}{n^2}e_{i,g+1}^{n-3}+\frac{1}{n}e_{i,g+2}^{n-3}\right]+\frac{1}{n}\sum_{\tau=1}^{n-4}e_{i,g+\tau+2}^{n-\tau-3} .
\end{align}
Subtract \eqref{eq:family1a} from \eqref{eq:family3a} to get
\begin{align}\label{eq:family3b}
& \begin{split}
&\left(-1+\frac{1}{n}\right)e_{i,g}^{n-1} +\frac{1}{n}e^{n-1}_{i,g+1}+\frac{1}{n}e_{i,g+2}^{n-2}-e_{i,g+1}^{n-2} \\
&+\left(1-\frac{1}{n}\right)\left[\frac{1}{n^2}e_{i,g+1}^{n-3} 
+\frac{1}{n}e_{i,g+2}^{n-3}\right] -e_{i,g+2}^{n-3}+\left(-1+\frac{1}{n}\right)\sum_{\tau=3}^{n-2}e_{i,g+\tau}^{n-1-\tau} .
\end{split}
\end{align}

\item[IV$^a$.] Pick the point \eqref{eq:family4} associated with $g+1\in J$ and any $s\in[n-4]$,
\begin{align}\label{eq:family4a}
\begin{split}
& \frac{1}{n}\sum_{k\in N}e_{k,g}^n +\sum_{\tau=0}^{n-s-2}\frac{1}{n}e^{n-\tau-1}_{i,g+\tau+1} \\
&+\left(1-\frac{1}{n}\right)\left[\sum_{\tau=0}^{n-s-2}\frac{1}{n^{n-\tau-s-1}}e^{s}_{i,g+\tau+1}\right]+\frac{1}{n}\sum_{\tau=1}^{s-1}e_{i,g+n-1-s+\tau}^{s-\tau} .
\end{split}
\end{align}
Subtract \eqref{eq:family1a} from \eqref{eq:family4a} to get
\begin{align*}
\sum_{\tau=0}^{n-s-2}\frac{1}{n}e^{n-\tau-1}_{i,g+\tau+1}&+\left(1-\frac{1}{n}\right)\left[\sum_{\tau=0}^{n-s-2}\frac{1}{n^{n-\tau-s-1}}e^{s}_{i,g+\tau+1}\right]  \\
&+\frac{1}{n}\sum_{\tau=1}^{s-1}e_{i,g+n-1-s+\tau}^{s-\tau}-\left(1-\frac{1}{n}\right)e_{i,g}^{n-1}-\sum_{\tau=1}^{n-2}e_{i,g+\tau}^{n-1-\tau},
\end{align*}
which is equivalent to
\begin{align}\label{eq:family4b}
\begin{split}
& \left(-1+\frac{1}{n}\right)e_{i,g}^{n-1} +\frac{1}{n}e^{n-1}_{i,g+1}+\sum_{\tau=1}^{n-s-2}\left[\frac{1}{n}e^{n-\tau-1}_{i,g+\tau+1}-e_{i,g+\tau}^{n-\tau-1}\right]  \\
&+\left(1-\frac{1}{n}\right)\left[\sum_{\tau=0}^{n-s-2}\frac{1}{n^{n-\tau-s-1}}e^{s}_{i,g+\tau+1}\right] -e_{i,g+n-s-1}^s +\left(-1+\frac{1}{n}\right)\sum_{\tau=n-s}^{n-2}e_{i,g+\tau}^{n-1-\tau}.
\end{split}
\end{align}
\end{enumerate}

Since $B_1$ is a diagonal matrix, for the rest of the proof we focus on the matrix $C$, formed by points in families II$^a$, III$^a$, IV$^a$ and V. Next, we apply Gaussian elimination on $C$.
\begin{enumerate}
\item[III$^b$.] Subtract \eqref{eq:family2b} from \eqref{eq:family3b},
\begin{align*}
\left(-1+\frac{1}{n}\right)e_{i,g}^{n-1}&+\frac{1}{n}e^{n-1}_{i,g+1}+\frac{1}{n}e_{i,g+2}^{n-2}-e_{i,g+1}^{n-2} +\left(1-\frac{1}{n}\right)\left[\frac{1}{n^2}e_{i,g+1}^{n-3}\right. \\
&\left.+\frac{1}{n}e_{i,g+2}^{n-3}\right] -e_{i,g+2}^{n-3}+\left(-1+\frac{1}{n}\right)\sum_{\tau=3}^{n-2}e_{i,g+\tau}^{n-1-\tau}\\
&-\frac{1}{n}e^{n-1}_{i,g+1}-\left(-1+\frac{1}{n}\right)e_{i,g}^{n-1}-\left(\frac{1}{n}-\frac{1}{n^2}-1\right)e_{i,g+1}^{n-2}\\
&-\left(-1+\frac{1}{n}\right)\sum_{\tau=2}^{n-2}e_{i,g+\tau}^{n-1-\tau} ,
\end{align*}
which is equivalent to
\begin{equation}\label{eq:family3bb}
\left(-\frac{1}{n}+\frac{1}{n^2}\right)e_{i,g+1}^{n-2}+\frac{1}{n}e_{i,g+2}^{n-2}+\left(\frac{1}{n^2}-\frac{1}{n^3}\right)e_{i,g+1}^{n-3}-\frac{1}{n^2}e_{i,g+2}^{n-3} .
\end{equation}

\item[IV$^b$.] For every $s\in[n-4]$, subtract \eqref{eq:family2b} from \eqref{eq:family4b},
\begin{align*}
\left(-1+\frac{1}{n}\right)e_{i,g}^{n-1}&+\frac{1}{n}e^{n-1}_{i,g+1}+\sum_{\tau=1}^{n-s-2}\left[\frac{1}{n}e^{n-\tau-1}_{i,g+\tau+1}-e_{i,g+\tau}^{n-\tau-1}\right]  \notag\\
&+\left(1-\frac{1}{n}\right)\left[\sum_{\tau=0}^{n-s-2}\frac{1}{n^{n-\tau-s-1}}e^{s}_{i,g+\tau+1}\right] -e_{i,g+n-s-1}^s\notag \\
&+\left(-1+\frac{1}{n}\right)\sum_{\tau=n-s}^{n-2}e_{i,g+\tau}^{n-1-\tau}\\
&-\frac{1}{n}e^{n-1}_{i,g+1}-\left(-1+\frac{1}{n}\right)e_{i,g}^{n-1}-\left(\frac{1}{n}-\frac{1}{n^2}-1\right)e_{i,g+1}^{n-2}\\
&-\left(-1+\frac{1}{n}\right)\sum_{\tau=2}^{n-2}e_{i,g+\tau}^{n-1-\tau} ,
\end{align*}
which is equivalent to
\begin{align}\label{eq:family4bb}
\left(-\frac{1}{n}+\frac{1}{n^2}\right)&e_{i,g+1}^{n-2}+\frac{1}{n}e_{i,g+2}^{n-2}+\frac{1}{n}\sum_{\tau=2}^{n-s-2}\left[e^{n-\tau-1}_{i,g+\tau+1}-e_{i,g+\tau}^{n-\tau-1}\right]  \notag\\
&+\left(1-\frac{1}{n}\right)\left[\sum_{\tau=0}^{n-s-3}\frac{1}{n^{n-\tau-s-1}}e^{s}_{i,g+\tau+1}\right] -\frac{1}{n^2}e_{i,g+n-s-1}^s .
\end{align}

\item[IV$^c$.] For every $s\in[n-4]$, subtract \eqref{eq:family3bb} from \eqref{eq:family4bb},
\begin{align*}
\left(-\frac{1}{n}+\frac{1}{n^2}\right)e_{i,g+1}^{n-2}&+\frac{1}{n}e_{i,g+2}^{n-2}+\frac{1}{n}\sum_{\tau=2}^{n-s-2}\left[e^{n-\tau-1}_{i,g+\tau+1}-e_{i,g+\tau}^{n-\tau-1}\right]  \notag\\
&+\left(1-\frac{1}{n}\right)\left[\sum_{\tau=0}^{n-s-3}\frac{1}{n^{n-\tau-s-1}}e^{s}_{i,g+\tau+1}\right] -\frac{1}{n^2}e_{i,g+n-s-1}^s\\
&-\left(-\frac{1}{n}+\frac{1}{n^2}\right)e_{i,g+1}^{n-2}-\frac{1}{n}e_{i,g+2}^{n-2}-\left(\frac{1}{n^2}-\frac{1}{n^3}\right)e_{i,g+1}^{n-3}+\frac{1}{n^2}e_{i,g+2}^{n-3} ,
\end{align*}
which is equivalent to
\begin{align}\label{eq:family4cc}
\frac{1}{n}e^{n-3}_{i,g+3}&-\left(\frac{1}{n^2}-\frac{1}{n^3}\right)e_{i,g+1}^{n-3}+\left(\frac{1}{n^2}-\frac{1}{n}\right)e_{i,g+2}^{n-3}+\frac{1}{n}\sum_{\tau=3}^{n-s-2}\left[e^{n-\tau-1}_{i,g+\tau+1}-e_{i,g+\tau}^{n-\tau-1}\right]  \notag\\
&+\left(1-\frac{1}{n}\right)\left[\sum_{\tau=0}^{n-s-3}\frac{1}{n^{n-\tau-s-1}}e^{s}_{i,g+\tau+1}\right] -\frac{1}{n^2}e_{i,g+n-s-1}^s .
\end{align}

\item[IV$^d$.] For every $s\in[n-5]$, subtract \eqref{eq:family4cc} corresponding to $s+1$ from \eqref{eq:family4cc} corresponding to $s$,
\begin{align*}
\frac{1}{n}e^{n-3}_{i,g+3}&-\left(\frac{1}{n^2}-\frac{1}{n^3}\right)e_{i,g+1}^{n-3}+\left(\frac{1}{n^2}-\frac{1}{n}\right)e_{i,g+2}^{n-3}+\frac{1}{n}\sum_{\tau=3}^{n-s-2}\left[e^{n-\tau-1}_{i,g+\tau+1}-e_{i,g+\tau}^{n-\tau-1}\right]  \notag\\
&+\left(1-\frac{1}{n}\right)\left[\sum_{\tau=0}^{n-s-3}\frac{1}{n^{n-\tau-s-1}}e^{s}_{i,g+\tau+1}\right] -\frac{1}{n^2}e_{i,g+n-s-1}^s\\
&-\frac{1}{n}e^{n-3}_{i,g+3}+\left(\frac{1}{n^2}-\frac{1}{n^3}\right)e_{i,g+1}^{n-3}-\left(\frac{1}{n^2}-\frac{1}{n}\right)e_{i,g+2}^{n-3}\\ &-\frac{1}{n}\sum_{\tau=3}^{n-s-3}\left[e^{n-\tau-1}_{i,g+\tau+1}-e_{i,g+\tau}^{n-\tau-1}\right]  \notag\\
&-\left(1-\frac{1}{n}\right)\left[\sum_{\tau=0}^{n-s-4}\frac{1}{n^{n-\tau-s-2}}e^{s+1}_{i,g+\tau+1}\right] +\frac{1}{n^2}e_{i,g+n-s-2}^{s+1} ,
\end{align*}
which is equivalent to
\begin{align}\label{eq:family4dd}
\frac{1}{n}&e^{s+1}_{i,g+n-s-1}+\left(\frac{1}{n^2}-\frac{1}{n}\right)e_{i,g+n-s-2}^{s+1}+\left(-1+\frac{1}{n}\right)\left[\sum_{\tau=0}^{n-s-4}\frac{1}{n^{n-\tau-s-2}}e^{s+1}_{i,g+\tau+1}\right] \notag\\
& \qquad -\frac{1}{n^2}e_{i,g+n-s-1}^s+\left(\frac{1}{n^{2}}-\frac{1}{n^{3}}\right)e^{s}_{i,g+n-s-2} \notag\\
& \qquad +\left(1-\frac{1}{n}\right)\left[\sum_{\tau=0}^{n-s-4}\frac{1}{n^{n-\tau-s-1}}e^{s}_{i,g+\tau+1}\right] .
\end{align}
For $s=n-4$, we do not need this step, since from \eqref{eq:family4cc} we have
\begin{align*}
\frac{1}{n}e^{n-3}_{i,g+3}&+\left(\frac{1}{n^2}-\frac{1}{n}\right)e_{i,g+2}^{n-3}+\left(-\frac{1}{n^2}+\frac{1}{n^3}\right)e_{i,g+1}^{n-3}\\
&-\frac{1}{n^2}e_{i,g+3}^{n-4}+\left(\frac{1}{n^2}-\frac{1}{n^3}\right)e^{n-4}_{i,g+2}+\left(\frac{1}{n^3}-\frac{1}{n^4}\right)e^{n-4}_{i,g+1} .
\end{align*}
Observe that for any $s\in[n-4]$ and $g\in J$, we can multiply row $(i,g,s+1)$ by $-1/n$ and we get the entry in row $(i,g,s)$.

\item[II$^b$.] Finally, pick a point \eqref{eq:family5} in family V for $g\in J$,
\begin{equation}\label{eq:family2aa}
\frac{1}{n}e_{i,g}^{n-1}+\frac{1}{n}e_{i,g+1}^{n-2}+\cdots+\frac{1}{n}e_{i,g+n-2}^1 .
\end{equation}
Now multiply \eqref{eq:family2aa} by $(1-n)$ and subtract it from \eqref{eq:family2b} for $g\in J$, yielding
\begin{align*}
\frac{1}{n}e^{n-1}_{i,g+1}+\left(-1+\frac{1}{n}\right)e_{i,g}^{n-1}&+\left(\frac{1}{n}-\frac{1}{n^2}-1\right)e_{i,g+1}^{n-2} \\
&+\left(-1+\frac{1}{n}\right)\sum_{\tau=2}^{n-2}e_{i,g+\tau}^{n-1-\tau}-\frac{1-n}{n}\sum_{\tau=0}^{n-2}e_{i,g+\tau}^{n-\tau-1} ,
\end{align*}
which is equivalent to
$$
\frac{1}{n}e^{n-1}_{i,g+1}-\frac{1}{n^2}e_{i,g+1}^{n-2}.
$$
Since we have $g+1$ in those two stages, we have a general expression for any $g\in J$,
\begin{equation}\label{eq:family2bb}
\frac{1}{n}e^{n-1}_{i,g}-\frac{1}{n^2}e_{i,g}^{n-2}.
\end{equation}
As before, we can multiply row $(i,g,n-1)$ by $-1/n$  to get the entry in row $(i,g,n-2)$.
\end{enumerate}

Now, we can organize the points in $C$ as
$$C=[\text{V}, \text{II}^b, \text{III}^b, \text{IV}^d_{n-4},\ldots, \text{IV}^d_s, \ldots, \text{IV}^d_1],$$
where $\text{IV}^d_s$ corresponds to the block of points ($g\in J$) with $s\in[n-4]$. $C$ has the form
$$C=
\begin{pmatrix}
C_{n-1} & D_{n-2} & 0 & 0& 0& \ldots &0&0 \\
C_{n-2} & -\frac{1}{n}D_{n-2} & D_{n-3} & 0 &0&\ldots&0&0\\
C_{n-3} & 0 & -\frac{1}{n}D_{n-3} & D_{n-4} &0&\ldots&0&0\\
\vdots & & &\ddots &&& \vdots\\
C_{2} & 0 & 0 & 0 &0&\ldots&-\frac{1}{n}D_{2}&D_{1}\\
C_{1} & 0 & 0 & 0 &0&\ldots&0&-\frac{1}{n}D_{1}\\
\end{pmatrix},$$
where every $C_i$ and $D_i$ are \emph{circulant} matrices \cite{kra_simanca12} of size $(n-1)\times(n-1)$. Since the determinant is invariant under elementary row and column operations, we can perform Gaussian elimination (of rows) from bottom to top, and we get
$$\bar{C}=
\begin{pmatrix}
\bar{C}_{n-1} & 0 & 0 & 0& 0& \ldots &0&0 \\
\bar{C}_{n-2} & -\frac{1}{n}D_{n-2} & 0 & 0 &0&\ldots&0&0\\
\bar{C}_{n-3} & 0 & -\frac{1}{n}D_{n-3} & 0 &0&\ldots&0&0\\
\vdots & & &\ddots &&& \vdots\\
\bar{C}_{2} & 0 & 0 & 0 &0&\ldots&-\frac{1}{n}D_{2}&0\\
C_{1} & 0 & 0 & 0 &0&\ldots&0&-\frac{1}{n}D_{1}\\
\end{pmatrix},$$
where
\begin{align*}
\bar{C}_1&=C_1+n(C_2-\cdots n(C_{n-2}+nC_{n-1})\cdots)\\
&=\circu(1/n,1,n,n^2,\ldots,n^{n-3}),\\
D_{n-2}&=\circu(1/n,0,\ldots0),\\
D_{n-3}&=\circu\left(\frac{1}{n},-\frac{1}{n}+\frac{1}{n^2},0,\ldots0\right)\\
&\vdots\\
D_s&=\circu\left(\frac{1}{n},-\frac{1}{n}+\frac{1}{n^2},-\frac{1}{n^2}+\frac{1}{n^3},\ldots,-\frac{1}{n^{n-s-2}}+\frac{1}{n^{n-s-1}},0,\ldots0\right)\\
&\vdots\\
D_1&=\circu\left(\frac{1}{n},-\frac{1}{n}+\frac{1}{n^2},-\frac{1}{n^2}+\frac{1}{n^3},\ldots,-\frac{1}{n^{n-3}}+\frac{1}{n^{n-2}}\right)
\end{align*}
For $\bar{C}_1$, the last entry, $n^{n-3}$, is greater than the sum of the remaining entries. The same applies for $D_s$, with entry $1/n$. Due to Proposition 18 in \cite{kra_simanca12} all these matrices are nonsingular, so $\bar{C}$ is nonsingular, and therefore $C$ is nonsingular. This implies that $A$ is nonsingular, showing that these points are linearly independent, and proceeding in the same way as we did in the proof of Proposition \ref{prop:ineq_Jsize_1} for the remaining components, we can show \eqref{eq:validExt2-n-1} is facet-defining.
\end{proof}

\begin{proof}[Proof of Theorem \ref{theo:ineq_IJ_gen}.]
The proof is similar to Theorem \ref{theo:ineq_Jsize_gen}. Again, assume $h=n-1$ and $t=1$; the remaining cases follow a similar construction of linearly independent points. Without loss of generality we can assume that $i_\tau =i$ for all $\tau \in [n-2]$. So, we have an inequality of the form
\begin{equation}\label{eq:validExt2-genI-n-1}
r Z_{N,J}^n+n Z_{I,J}^{n-1}+\sum_{\tau=1}^{n-2} n^{n-\tau} Z_{i,J}^{\tau}\leq r+\sum_{\tau=1}^{n-2} n^{\tau}.
\end{equation}
We construct the following linearly independent points.
\begin{enumerate}
\item[I.] Fix $k\in N$ and $j\in J$. In stage $n$, if node $k$ appears, match it to node $j$, with probability $1/n$. For the remaining stages match according to $(j,j+1,\ldots,j+n-2) \in \circu(J)$. In terms of probability, if any $k'\in I$ appears in stage $n-1$, then it is matched to $j$ with probability $(1-1/n)\cdot 1/n$, so the probability of matching $j$ in stage $n-1$ is $(1-1/n)\cdot r/n$. For the rest, the probability is $1/n$. So, we have the point
 \begin{equation}\label{eq:family1teo2}\frac{1}{n}e_{k,j}^n+\frac{1}{n}\left(1-\frac{1}{n}\right)\sum_{k'\in I}e_{k',j}^{n-1}+\frac{1}{n}\sum_{\tau=1}^{n-2}e_{i,j+\tau}^{n-1-\tau}.\end{equation}
By a simple calculation, it is easy to see that each of these points achieves the right-hand side of \eqref{eq:validExt2-genI-n-1}. Since we chose any arbitrary $k\in N$ and $j\in J$, we have $n(n-1)$ points in this family.

\item[II.] Fix $j\in J$ and $k\in I$. In this family we repeat the same ad to match in stages $n-1$ and $n-2$. 
If $k$ appears in stage $n-1$, match it to $j$ with probability $1/n$. Then, if $i$ appears in stage $n-2$ and $k$ did not appear in $ n-1 $, match it to $j$ with probability $(1-1/n)\cdot 1/n$. For the remaining stages match according to $(j+n-2,j,j+1,\ldots,j+n-3) \in \circu(J)$; in stage $n$ match any $k'\in N$ that appears with $j+n-2$, in stage $n-3$ match $i$ to $ j + 1 $ if it appears, and so on. So, we have the point
\begin{equation}\label{eq:family2teo2}
\frac{1}{n}\sum_{k'\in N}e_{k',j+n-2}^n+\frac{1}{n}e^{n-1}_{k,j}+\frac{1}{n}\left(1-\frac{1}{n}\right)e_{i,j}^{n-2}+\frac{1}{n}\sum_{\tau=1}^{n-3}e_{i,j+\tau}^{n-\tau-2} .
\end{equation}
By a simple calculation, we get the right-hand side of \eqref{eq:validExt2-n-1}. Finally, since we chose an arbitrary $j\in J$ and $k'\in I$, we have $r(n-1)$ points in this family.

\item[V.]  Fix $j\in J$. In this family we do not match in stage $n$, and in the remaining stages we match according to a vector in $\circu(J)$. If any $k\in I$ appears in stage $n-1$, match it to $j$ with probability $1/n$, if $i$ appears in stage $n-2$, match it to node $j+1$, and so forth. 
    So we have the point
\begin{equation}\label{eq:family5teo2}
\frac{1}{n}\sum_{k\in I}e_{k,j}^{n-1}+\frac{1}{n}\sum_{\tau=1}^{n-2}e_{i,j+\tau}^{n-\tau-1} .
\end{equation}
By a simple calculation, we get the right-hand side of \eqref{eq:validExt2-n-1}. Since we chose an arbitrary $j\in J$, we have $n-1$ points in this family.
\end{enumerate}
Families III, and IV remain the same as in the proof of Theorem \ref{theo:ineq_Jsize_gen}, so in total we have $(n-1)(r+2n-2)$ points. The rest of the proof follows the same argument as Theorem \ref{theo:ineq_Jsize_gen}.
\end{proof}

\begin{proof}[Proof of Theorem \ref{theo:ineq_general}.]
The proof follows the same argument as the previous two theorems, the only difference being that the collection of points given by policies corresponding to $I$ is bigger; however, they still form a diagonal block, so we can apply the same procedure.
\end{proof}

\subsection{Detailed Experiment Results}

 \begin{table}[ht]
\begin{center}
{\footnotesize
 {\begin{tabular}{|c|c|c|c|c|c|c|c|}
\hline
Instance   & \eqref{eq:obm_relax_prev} + \eqref{eq:right-star_prev} & \eqref{eq:prob+J1LP} & \eqref{eq:prob+J1LP} + \eqref{eq:ineq_Jsize_2} &  Exp.\ Matching & \eqref{eq:DP}     & \eqref{eq:time_index_policy} & TDR Policy \cite{torrico_etal18} \\ \hline
S1  & 8.9613 & 8.7261 & 8.7170 & 8.5818 & 8.3976 & 8.3562 & 8.3452 \\ \hline
S2  & 8.0125 & 7.7354 & 7.7147 & 7.4635 & 7.4055 & 7.4095 & 7.3980 \\ \hline
S3  & 7.0252 & 6.8849 & 6.8795 & 6.8109 & 6.7241 & 6.7239 & 6.6546 \\ \hline
S4  & 8.4799 & 8.1415 & 8.1282 & 8.1360 & 7.8468 & 7.8142 & 7.7895 \\ \hline
S5  & 7.9805 & 7.4972 & 7.4654 & 7.3191 & 7.0849 & 7.0590 & 6.8223 \\ \hline
S6  & 9.2650 & 8.5601 & 8.5289 & 8.4907 & 8.1364 & 8.1057 & 8.0802 \\ \hline
S7  & 7.1327 & 6.9320 & 6.9155 & 6.7885 & 6.6895 & 6.6699 & 6.6709 \\ \hline
S8  & 8.2993 & 7.8525 & 7.8181 & 7.6588 & 7.4762 & 7.4631 & 7.4203 \\ \hline
S9  & 7.0193 & 6.6523 & 6.6320 & 6.4263 & 6.3451 & 6.3364 & 6.2574 \\ \hline
S10 & 7.1206 & 6.9756 & 6.9700 & 6.9593 & 6.8103 & 6.7790 & 6.7588 \\ \hline
S11 & 8.9684 & 8.5841 & 8.5519 & 8.3237 & 8.1223 & 8.0996 & 7.9973 \\ \hline
S12 & 6.7110 & 6.4682 & 6.4537 & 6.3265 & 6.2555 & 6.2542 & 6.2542 \\ \hline
S13 & 7.8013 & 7.5626 & 7.5543 & 7.7048 & 7.3845 & 7.3547 & 7.3122 \\ \hline
S14 & 8.0801 & 7.7047 & 7.6880 & 7.5756 & 7.4003 & 7.3787 & 7.2813 \\ \hline
S15 & 8.2220 & 8.0324 & 7.9922 & 7.9251 & 7.7369 & 7.7026 & 7.7021 \\ \hline
S16 & 9.3714 & 8.8684 & 8.8560 & 8.9708 & 8.4948 & 8.4672 & 8.4034 \\ \hline
S17 & 6.9684 & 6.8016 & 6.7739 & 6.6437 & 6.5785 & 6.5717 & 6.4977 \\ \hline
S18 & 9.3936 & 8.8414 & 8.8248 & 8.8218 & 8.4484 & 8.4247 & 8.3458 \\ \hline
S19 & 7.0304 & 6.7305 & 6.7138 & 6.5861 & 6.4776 & 6.4778 & 6.4772 \\ \hline
S20 & 8.7322 & 8.4326 & 8.4139 & 8.2763 & 8.0769 & 8.0485 & 7.9651 \\ \hline
\end{tabular}}
 \caption{Experiment results for small instances.}\label{tab:small}}
\end{center}
\end{table}

 \begin{table}[ht]
\begin{center}
{\footnotesize
    {\begin{tabular}{|c|c|c|c|c|c|}
\hline
Instance   & \eqref{eq:obm_relax_prev} + \eqref{eq:right-star_prev} & \eqref{eq:prob+J1LP}  &   Exp.\ Matching & 
 \eqref{eq:time_index_policy} & TDR Policy \cite{torrico_etal18} \\ \hline
LD1  & 99.8550 & 97.0381 & 99.8153 & 94.8957 & 94.8471 \\ \hline
LD2  & 99.8035 & 97.1728 & 99.7964 & 95.5454 & 95.0983 \\ \hline
LD3  & 99.9368 & 97.2471 & 99.8325 & 95.4772 & 95.1928 \\ \hline
LD4  & 99.9083 & 97.3543 & 99.8998 & 95.7123 & 95.4918 \\ \hline
LD5  & 99.9135 & 97.1967 & 99.9033 & 95.5889 & 95.3354 \\ \hline
LD6  & 99.9533 & 97.3301 & 99.8303 & 95.5145 & 95.2998 \\ \hline
LD7  & 99.9528 & 97.3809 & 99.9314 & 95.7929 & 95.3830 \\ \hline
LD8  & 99.9100 & 97.3675 & 99.9006 & 95.7148 & 95.3640 \\ \hline
LD9  & 99.7849 & 97.1600 & 99.7414 & 95.2710 & 95.0698 \\ \hline
LD10 & 99.8933 & 97.2112 & 99.8805 & 95.6271 & 95.2136 \\ \hline
LD11 & 99.9175 & 97.3080 & 99.8908 & 95.6526 & 95.1530 \\ \hline
LD12 & 99.8668 & 97.3086 & 99.8597 & 95.5097 & 95.5914 \\ \hline
LD13 & 99.7805 & 97.1135 & 99.6728 & 94.9733 & 94.6578 \\ \hline
LD14 & 99.8641 & 97.3257 & 99.7695 & 95.3197 & 95.4821 \\ \hline
LD15 & 99.9079 & 97.1715 & 99.8723 & 95.2706 & 95.2057 \\ \hline
LD16 & 99.7806 & 96.8779 & 99.7579 & 94.8829 & 94.6222 \\ \hline
LD17 & 99.7620 & 97.2239 & 99.7554 & 95.5676 & 95.4672 \\ \hline
LD18 & 99.9504 & 97.3627 & 99.8417 & 95.7888 & 95.5311 \\ \hline
LD19 & 99.9603 & 97.4715 & 99.9530 & 95.8706 & 95.7686 \\ \hline
LD20 & 99.9392 & 97.1932 & 99.9263 & 95.2861 & 95.1938 \\ \hline
\end{tabular}}
\caption{Experiment results for large, dense instances.}\label{tab:ld}}
\end{center}
\end{table}

 \begin{table}[ht]
\begin{center}
{\footnotesize
{\begin{tabular}{|c|c|c|c|c|c|c|}
\hline
Instance   & \eqref{eq:obm_relax_prev} + \eqref{eq:right-star_prev} & \eqref{eq:prob+J1LP} & \eqref{eq:prob+J1LP} + \eqref{eq:ineq_Jsize_2} &   Exp.\ Matching & 
\eqref{eq:time_index_policy} & TDR Policy \cite{torrico_etal18} \\ \hline
LS1  & 79.1015 & 74.7635 & 74.7260 & 72.3904 & 69.7186 & 69.0988 \\ \hline
LS2  & 79.0703 & 75.3885 & 75.3555 & 73.3026 & 70.8925 & 70.2621 \\ \hline
LS3  & 78.0157 & 74.3402 & 74.3051 & 72.0072 & 69.6943 & 69.1178 \\ \hline
LS4  & 78.6983 & 74.6555 & 74.6181 & 72.3612 & 69.7633 & 69.1302 \\ \hline
LS5  & 81.8877 & 76.7406 & 76.7085 & 74.9893 & 71.6122 & 70.7586 \\ \hline
LS6  & 81.4750 & 77.6830 & 77.6517 & 75.5756 & 72.8110 & 72.3019 \\ \hline
LS7  & 84.0754 & 79.2133 & 79.1825 & 77.4263 & 73.6374 & 72.3566 \\ \hline
LS8  & 71.9270 & 69.5745 & 69.5391 & 67.6728 & 65.7709 & 65.4557 \\ \hline
LS9  & 77.0648 & 73.8295 & 73.7947 & 71.4234 & 69.3153 & 68.5672 \\ \hline
LS10 & 82.5481 & 77.2318 & 77.2028 & 75.3500 & 72.2398 & 71.4496 \\ \hline
LS11 & 73.6532 & 70.8502 & 70.8139 & 68.2690 & 66.2972 & 64.9590 \\ \hline
LS12 & 81.1134 & 77.2298 & 77.1911 & 74.8116 & 71.9705 & 71.3263 \\ \hline
LS13 & 75.2453 & 72.2625 & 72.2297 & 70.0278 & 67.6604 & 66.7906 \\ \hline
LS14 & 80.7172 & 75.8009 & 75.7654 & 74.1015 & 70.7923 & 70.3709 \\ \hline
LS15 & 74.7313 & 72.2967 & 72.2627 & 69.7801 & 68.0195 & 67.1434 \\ \hline
LS16 & 75.9530 & 72.2955 & 72.2691 & 70.2883 & 67.5794 & 66.6571 \\ \hline
LS17 & 78.9469 & 74.2399 & 74.2093 & 72.6022 & 69.4289 & 68.6303 \\ \hline
LS18 & 79.2340 & 74.8535 & 74.8261 & 73.3324 & 70.1805 & 69.6408 \\ \hline
LS19 & 78.6156 & 75.6028 & 75.5745 & 73.7973 & 71.2080 & 70.6332 \\ \hline
LS20 & 82.8804 & 78.4299 & 78.3911 & 76.6972 & 73.1519 & 72.9224 \\ \hline
\end{tabular}}
    \caption{Experiment results for large, sparse instances.}\label{tab:ls}}
\end{center}
\end{table}


%
%
%

\end{document}